\documentclass[12pt]{amsart} 
\usepackage{graphicx,color} 
 \usepackage{amsmath}
  \usepackage{amssymb,bbm}
  \usepackage{amscd}
\usepackage{enumerate}
\usepackage{multirow}
\usepackage{hyperref}
\usepackage{tikz}
\usetikzlibrary{arrows}
\usepackage{comment}

\renewcommand{\L}{\mathcal L}

\newcommand{\des}{\operatorname{des}}

\newcommand{\be}{\begin{equation}}
\newcommand{\ee}{\end{equation}}
\newcommand{\ds}{\displaystyle}

\newcommand{\im}{\operatorname{im}}
\newcommand{\fiber}{\operatorname{fiber}}
\newcommand{\LL}{\mathcal{L}}
\newcommand{\LM}{L^{\Monoid}}
\newcommand{\R}{\mathcal{R}}
\newcommand{\rfactor}{\operatorname{rfactor}}
\newcommand{\Rfactor}{\operatorname{Rfactor}}
\newcommand{\supp}{\operatorname{supp}}
\newcommand{\Monoid}{\mathcal{M}}
\newcommand{\prom}{\hat\partial}

\newtheorem{thm}{Theorem}[section]
\newtheorem{cor}[thm]{Corollary}
\newtheorem{lem}[thm]{Lemma}
\newtheorem{prop}[thm]{Proposition}
\newtheorem{defn}[thm]{Definition}
\newtheorem{rem}[thm]{Remark}

\newtheorem{eg}[thm]{Example}

\numberwithin{equation}{section}

\makeatletter
\def\Ddots{\mathinner{\mkern1mu\raise\p@
\vbox{\kern7\p@\hbox{.}}\mkern2mu
\raise4\p@\hbox{.}\mkern2mu\raise7\p@\hbox{.}\mkern1mu}}
\makeatother

\begin{document} 
\title{Combinatorial Markov chains on linear extensions}

\author[A. Ayyer]{Arvind Ayyer}
\address[Arvind Ayyer]{Department of Mathematics, UC Davis, One Shields Ave., Davis, CA 95616-8633, U.S.A. \newline
New address: Department of Mathematics, Department of Mathematics, Indian Institute of Science, Bangalore - 560012, India.
}

\email{arvind@math.iisc.ernet.in}

\author[S. Klee]{Steven Klee}
\address[Steven Klee]{Department of Mathematics, UC Davis, One Shields Ave., Davis, CA 95616-8633, U.S.A.\newline
New address: Department of Mathematics, Seattle University, 901 12th Avenue, Seattle, WA 98122-1090, U.S.A.
}
\email{klees@seattleu.edu}

\author[A. Schilling]{Anne Schilling}
\address[Anne Schilling]{Department of Mathematics, UC Davis, One Shields Ave., Davis, CA 95616-8633, U.S.A.}
\email{anne@math.ucdavis.edu}

\thanks{A.A. would like to acknowledge support from MSRI, where part of this work was done.
S.K. was supported by NSF VIGRE grant DMS--0636297. 
A.S. was supported by NSF grant DMS--1001256.}

\subjclass{Primary 06A07, 20M32, 20M30, 60J27; Secondary: 47D03}

\begin{abstract}
We consider generalizations of Sch\"utzenberger's promotion operator
on the set $\mathcal{L}$ of linear extensions of a finite poset of
size $n$. This gives rise to a strongly connected graph on
$\mathcal{L}$. By assigning weights to the edges of the graph in 
two different ways, we study two Markov chains, both of which are
irreducible. The stationary state of one gives rise to the uniform
distribution, whereas the weights of the stationary state of the other
has a nice product formula. This generalizes results by Hendricks on
the Tsetlin library, which corresponds to the case when the poset is
the anti-chain and hence $\mathcal{L}=S_n$ is the full symmetric
group. We also provide explicit eigenvalues of the
transition matrix in general when the poset is a rooted forest.
This is shown by proving that the associated monoid is $\R$-trivial
and then using Steinberg's extension of Brown's theory for Markov chains on 
left regular bands to $\R$-trivial monoids.
\end{abstract}

\date{\today}  

\maketitle

\section{Introduction}

Sch\"utzenberger~\cite{schuetzenberger.1972} introduced the notion of
evacuation and promotion on the set of linear extensions of a finite
poset $P$ of size $n$. This generalizes promotion on standard Young
tableaux defined in terms of jeu-de-taquin
moves. Haiman~\cite{haiman.1992} as well as Malvenuto and
Reutenauer~\cite{malvenuto_reutenauer.1994} simplified
Sch\"utzenberger's approach by expressing the promotion operator
$\partial$ in terms of more fundamental operators $\tau_i$ ($1\le
i<n$), which either act as the identity or as a simple
transposition. A beautiful survey on this subject was written by
Stanley~\cite{stanprom}.

In this paper, we consider a slight generalization of the promotion
operator defined as $\partial_i = \tau_i \tau_{i+1} \cdots \tau_{n-1}$
for $1\le i\le n$ with $\partial_1=\partial$ being the original
promotion operator. Since the operators $\partial_i$ act on the set of
all linear extensions of $P$, denoted $\mathcal{L}(P)$, this gives
rise to a graph whose vertices are the linear extensions and edges are
labeled by the action of $\partial_i$. We show that this graph is
strongly connected (see Proposition~\ref{proposition.strongly_connected}). 
As a result we obtain two irreducible Markov chains on
$\mathcal{L}(P)$ by assigning weights to the edges in two different
ways. In one case, the stationary state is uniform, that is, every
linear extension is equally likely to occur (see Theorem~\ref{theorem.uniform promotion}). 
In the other case, we obtain a nice product formula for the weights of the 
stationary distribution (see Theorem~\ref{theorem.promotion}).
We also consider  analogous Markov chains for the
adjacent transposition operators $\tau_i$, and give a combinatorial
formula for their stationary distributions (see Theorems~\ref{theorem.uniform transposition}
and~\ref{theorem.transposition}).

Our results can be viewed as a natural generalization of the results
of Hendricks~\cite{hendricks1,hendricks2} on the Tsetlin library
\cite{tsetlin}, which is a model for the way an arrangement of books
in a library shelf evolves over time. It is a Markov chain on
permutations, where the entry in the $i$th position is moved to the
front (or back depending on the conventions) with probability $p_i$. Hendricks' results from our viewpoint
correspond to the case when $P$ is an anti-chain and hence
$\mathcal{L}(P)=S_n$ is the full symmetric group.  Many variants of
the Tsetlin library have been studied and there is a wealth of
literature on the subject. We refer the interested reader to the
monographs by Letac~\cite{letac} and by Dies~\cite{dies}, as well as
the comprehensive bibliographies in~\cite{fill.1996} 
and~\cite{bidigare_hanlon_rockmore.1999}.

One of the most interesting properties of the Tsetlin library Markov
chain is that the eigenvalues of the transition matrix can be computed
exactly. The exact form of the eigenvalues was independently
investigated by several groups. Notably
Donnelly~\cite{donnelly.1991}, Kapoor and
Reingold~\cite{kapoor_reingold.1991}, and
Phatarfod~\cite{phatarfod.1991} studied the approach to stationarity
in great detail.
There has been some interest in finding exact formulas for the
eigenvalues for generalizations of the Tsetlin library. The first
major achievement in this direction was to interpret these results in
the context of hyperplane arrangements \cite{bidigare.1997,
  bidigare_hanlon_rockmore.1999, brown_diaconis.1998}.  This was
further generalized to a class of monoids called left regular
bands~\cite{brown.2000} and subsequently to all
bands~\cite{brown.2004} by Brown. This theory has been used
effectively by Bj\"orner~\cite{bjorner.2008, bjorner.2009} to extend
eigenvalue formulas on the Tsetlin library from a single shelf to
hierarchies of libraries.

In this paper we give explicit combinatorial formulas for the
eigenvalues and multiplicities for the transition matrix of the
promotion Markov chain when the underlying poset is a rooted forest
(see Theorem~\ref{theorem.eigenvalues}). This is
achieved by proving that the associated monoid is $\R$-trivial and
then using a generalization of Brown's theory~\cite{brown.2000} of
Markov chains for left regular bands to the $\R$-trivial case using
results by Steinberg~\cite{steinberg.2006, steinberg.2008}.

Computing the number of linear extensions is an important problem for
real world applications \cite{karzanov_khachiyan.1991}.  For example,
it relates to sorting algorithms in computer science, rankings in the
social sciences, and efficiently counting standard Young tableaux in
combinatorics. A recursive formula was given
in~\cite{edelman_hibi_stanley.1989}. Brightwell and
Winkler~\cite{brightwell_winkler.1991} showed that counting the number
of linear extensions is $\# P$-complete. Bubley and
Dyer~\cite{bubley.dyer.1999} provided an algorithm to (almost)
uniformly sample the set of linear extensions of a finite poset quickly.
We propose new Markov chains for sampling linear extensions uniformly randomly.
Further details are discussed in Section~\ref{section.outlook}.

The paper is outlined as follows.  In Section~\ref{section.promotion}
we define the extended promotion operator and investigate some of its
properties.  The extended promotion and transposition operators are
used in Section~\ref{section.markov chains} to define various Markov
chains, whose properties are studied in
Section~\ref{section.properties}. We also prove formulas for the
stationary distributions and explain the connection with the Tsetlin
library there.  In Section~\ref{section.chains} we derive the
partition function for the promotion Markov chains for rooted forests
as well as all eigenvalues together with their multiplicities of the
transition matrix.  The statements about eigenvalues and
multiplicities are proven in Section~\ref{section.R trivial} using the
theory of $\R$-trivial monoids.  We end with possible directions for
future research in Section~\ref{section.outlook}.  In
Appendix~\ref{section.appendix} we provide details about
implementations of linear extensions, Markov chains, and their
properties in {\tt Sage}~\cite{sage, sage-combinat} and {\tt Maple}.

\subsection*{Acknowledgements}
We would like to thank Richard Stanley for valuable input during his
visit to UC Davis in January 2012, Jes\'us De Loera, Persi Diaconis,
Franco Saliola, Benjamin Steinberg, and Peter Winkler for helpful
discussions.  Special thanks go to Nicolas M. Thi\'ery for his help
getting our code related to this project into {\tt
  Sage}~\cite{sage,sage-combinat}, for his discussions on the
representation theory of monoids, and for pointing out that
Theorem~\ref{theorem.eigenvalues} holds not only for unions of chains
but for rooted forests. John Stembridge's {\tt posets} package proved
very useful for computer experimentation.

\section{Extended promotion on linear extensions}
\label{section.promotion}

\subsection{Definition of extended promotion} \label{subsection.def prom}
Let  $P$ be an arbitrary poset of size $n$, with partial order denoted by $\preceq$. We assume that the 
elements of $P$ are labeled by integers in $[n]:=\{1,2,\ldots,n\}$. In addition, we assume
that the poset is naturally labeled, that is if $i,j \in P$ with $i \preceq j$ in $P$ then $i \le j$ as integers.
Let $\L:=\L(P)$ be the set of its {\bf linear extensions}, 
\be
\L(P) = \{ \pi \in S_{n} \mid i \prec j \text{ in $P$ } \implies \pi^{-1}_{i} < \pi^{-1}_{j} \text{ as integers} \},
\ee
which is naturally interpreted as a subset of the symmetric group $S_n$. Note that the identity permutation $e$ 
always belongs to $\L$. Let $P_{j}$ be the natural (induced) subposet of $P$ consisting of elements $k$ 
such that $j \preceq k$~\cite{stanenum}.

We now briefly recall the idea of {\bf promotion} of a linear extension of a poset $P$. Start with a linear extension 
$\pi \in \L(P)$ and imagine placing the label $\pi^{-1}_{i}$ in $P$ at the location $i$. By the definition of the linear 
extension, the labels will be well-ordered. The action of promotion of $\pi$ will give another linear extension of
$P$ as follows:
\begin{enumerate} 

\item The process starts with a seed, the label 1. First remove it and replace it by the minimum of all the labels covering it, $i$, say. 
 
\item Now look for the minimum of all labels covering $i$ in the original poset, and replace it, and continue in this way. 

\item This process ends when a label is a ``local maximum.'' Place the label $n+1$ at that point.

\item Decrease all the labels by 1.

\end{enumerate} 

This new linear extension is denoted $\pi \partial$  \cite{stanprom}.

\begin{eg}
\label{example.promotion slide}
Figure~\ref{figure:promotion-example} shows a poset (left) to which we assign the identity linear extension
$\pi = 123456789$. The linear extension $\pi'=\pi\partial = 214537869$ obtained by applying the promotion operator 
is depicted on the right. Note that indeed we place $\pi_i^{'-1}$ in position $i$, namely 3 is in position 5 in $\pi'$, so that 
5 in $\pi \partial$ is where 3 was originally.

\begin{center}
\begin{figure}[h]
\begin{tabular}{p{2in}p{2in}}
\begin{tikzpicture}[>=latex,line join=bevel,]
\node (1) at (66bp,7bp) [draw, draw = none] {$1$};
  \node (3) at (36bp,57bp) [draw, draw = none] {$3$};
  \node (2) at (36bp,7bp) [draw, draw = none] {$2$};
  \node (5) at (66bp,107bp) [draw, draw = none] {$5$};
  \node (4) at (66bp,57bp) [draw, draw = none] {$4$};
  \node (7) at (6bp,107bp) [draw, draw = none] {$7$};
  \node (6) at (36bp,107bp) [draw, draw = none] {$6$};
  \node (9) at (21bp,157bp) [draw, draw = none] {$9$};
  \node (8) at (96bp,107bp) [draw, draw = none] {$8$};
  \draw [black,-] (3) ..controls (50.771bp,32.382bp) and (57.698bp,20.837bp)  .. (1);
  \draw [black,-] (9) ..controls (28.603bp,131.66bp) and (31.916bp,120.61bp)  .. (6);
  \draw [black,-] (5) ..controls (66bp,81.179bp) and (66bp,70.462bp)  .. (4);
  \draw [black,-] (7) ..controls (20.771bp,82.382bp) and (27.698bp,70.837bp)  .. (3);
  \draw [black,-] (3) ..controls (36bp,31.179bp) and (36bp,20.462bp)  .. (2);
  \draw [black,-] (4) ..controls (66bp,31.179bp) and (66bp,20.462bp)  .. (1);
  \draw [black,-] (9) ..controls (13.397bp,131.66bp) and (10.084bp,120.61bp)  .. (7);
  \draw [black,-] (8) ..controls (81.229bp,82.382bp) and (74.302bp,70.837bp)  .. (4);
  \draw [black,-] (6) ..controls (36bp,81.179bp) and (36bp,70.462bp)  .. (3);
\end{tikzpicture}%
&
\begin{tikzpicture}[>=latex,line join=bevel,]
\node (1) at (66bp,7bp) [draw, draw = none] {$2$};
  \node (3) at (36bp,57bp) [draw, draw = none] {$5$};
  \node (2) at (36bp,7bp) [draw, draw = none] {$1$};
  \node (5) at (66bp,107bp) [draw, draw = none] {$4$};
  \node (4) at (66bp,57bp) [draw, draw = none] {$3$};
  \node (7) at (6bp,107bp) [draw, draw = none] {$6$};
  \node (6) at (36bp,107bp) [draw, draw = none] {$8$};
  \node (9) at (21bp,157bp) [draw, draw = none] {$9$};
  \node (8) at (96bp,107bp) [draw, draw = none] {$7$};
  \draw [black,-] (4) ..controls (66bp,31.179bp) and (66bp,20.462bp)  .. (1);
  \draw [black,-] (9) ..controls (28.603bp,131.66bp) and (31.916bp,120.61bp)  .. (6);
  \draw [black,-] (5) ..controls (66bp,81.179bp) and (66bp,70.462bp)  .. (4);
  \draw [black,-] (3) ..controls (50.771bp,32.382bp) and (57.698bp,20.837bp)  .. (1);
  \draw [black,-] (7) ..controls (20.771bp,82.382bp) and (27.698bp,70.837bp)  .. (3);
  \draw [black,-] (3) ..controls (36bp,31.179bp) and (36bp,20.462bp)  .. (2);
  \draw [black,-] (9) ..controls (13.397bp,131.66bp) and (10.084bp,120.61bp)  .. (7);
  \draw [black,-] (8) ..controls (81.229bp,82.382bp) and (74.302bp,70.837bp)  .. (4);
  \draw [black,-] (6) ..controls (36bp,81.179bp) and (36bp,70.462bp)  .. (3);
\end{tikzpicture}
\end{tabular}
\caption{A linear extension $\pi$ (left) and $\pi\partial$ (right).}
\label{figure:promotion-example}
\end{figure}
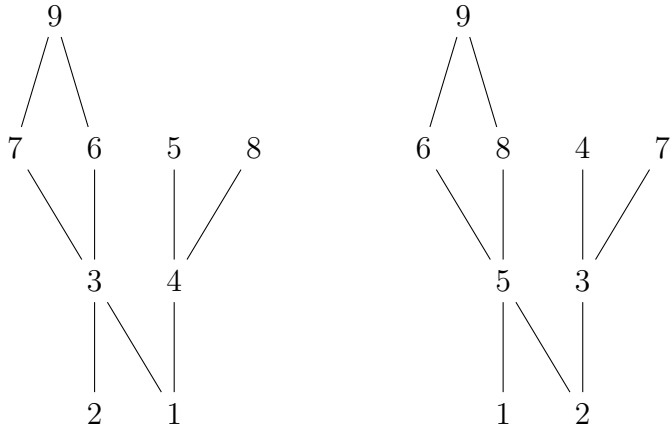
\end{center}

\noindent
Figure \ref{figure:promotion-example2} illustrates the steps used to construct the linear extension $\pi\partial$ from the linear extension $\pi$ from Figure \ref{figure:promotion-example}. Appendix~\ref{section.appendix} includes Sage implementation of this action.

\begin{center}
\begin{figure}[h]
\begin{tabular}{|p{1.75in}|p{1.75in}|p{1.75in}|} \hline
\textbf{Step 1:} Remove the minimal element $1$. &
\textbf{Step 2:}    The minimal element that covered $1$ was $3$, so replace $1$ with $3$. &
\textbf{Step 2 (continued):}  The minimal element that covered $3$ was $6$, so replace $3$ with $6$. \\ \hline
\begin{center}
\begin{tikzpicture}[>=latex,line join=bevel,]
\node (1) at (66bp,7bp) [draw, draw = none] {};
  \node (3) at (36bp,57bp) [draw, draw = none] {$3$};
  \node (2) at (36bp,7bp) [draw, draw = none] {$2$};
  \node (5) at (66bp,107bp) [draw, draw = none] {$5$};
  \node (4) at (66bp,57bp) [draw, draw = none] {$4$};
  \node (7) at (6bp,107bp) [draw, draw = none] {$7$};
  \node (6) at (36bp,107bp) [draw, draw = none] {$6$};
  \node (9) at (21bp,157bp) [draw, draw = none] {$9$};
  \node (8) at (96bp,107bp) [draw, draw = none] {$8$};
  \draw [black,-] (4) ..controls (66bp,31.179bp) and (66bp,20.462bp)  .. (1);
  \draw [black,-] (9) ..controls (28.603bp,131.66bp) and (31.916bp,120.61bp)  .. (6);
  \draw [black,-] (5) ..controls (66bp,81.179bp) and (66bp,70.462bp)  .. (4);
  \draw [black,-] (3) ..controls (50.771bp,32.382bp) and (57.698bp,20.837bp)  .. (1);
  \draw [black,-] (7) ..controls (20.771bp,82.382bp) and (27.698bp,70.837bp)  .. (3);
  \draw [black,-] (3) ..controls (36bp,31.179bp) and (36bp,20.462bp)  .. (2);
  \draw [black,-] (9) ..controls (13.397bp,131.66bp) and (10.084bp,120.61bp)  .. (7);
  \draw [black,-] (8) ..controls (81.229bp,82.382bp) and (74.302bp,70.837bp)  .. (4);
  \draw [black,-] (6) ..controls (36bp,81.179bp) and (36bp,70.462bp)  .. (3);
\end{tikzpicture}%
\end{center}

&
\begin{center}
\begin{tikzpicture}[>=latex,line join=bevel,]
\node (1) at (66bp,7bp) [draw, draw = none] {$3$};
  \node (3) at (36bp,57bp) [draw, draw = none] {};
  \node (2) at (36bp,7bp) [draw, draw = none] {$2$};
  \node (5) at (66bp,107bp) [draw, draw = none] {$5$};
  \node (4) at (66bp,57bp) [draw, draw = none] {$4$};
  \node (7) at (6bp,107bp) [draw, draw = none] {$7$};
  \node (6) at (36bp,107bp) [draw, draw = none] {$6$};
  \node (9) at (21bp,157bp) [draw, draw = none] {$9$};
  \node (8) at (96bp,107bp) [draw, draw = none] {$8$};
  \draw [black,-] (4) ..controls (66bp,31.179bp) and (66bp,20.462bp)  .. (1);
  \draw [black,-] (9) ..controls (28.603bp,131.66bp) and (31.916bp,120.61bp)  .. (6);
  \draw [black,-] (5) ..controls (66bp,81.179bp) and (66bp,70.462bp)  .. (4);
  \draw [black,-] (3) ..controls (50.771bp,32.382bp) and (57.698bp,20.837bp)  .. (1);
  \draw [black,-] (7) ..controls (20.771bp,82.382bp) and (27.698bp,70.837bp)  .. (3);
  \draw [black,-] (3) ..controls (36bp,31.179bp) and (36bp,20.462bp)  .. (2);
  \draw [black,-] (9) ..controls (13.397bp,131.66bp) and (10.084bp,120.61bp)  .. (7);
  \draw [black,-] (8) ..controls (81.229bp,82.382bp) and (74.302bp,70.837bp)  .. (4);
  \draw [black,-] (6) ..controls (36bp,81.179bp) and (36bp,70.462bp)  .. (3);
\end{tikzpicture}%
\end{center}
&
\begin{center}
\begin{tikzpicture}[>=latex,line join=bevel,]
\node (1) at (66bp,7bp) [draw, draw = none] {$3$};
  \node (3) at (36bp,57bp) [draw, draw = none] {$6$};
  \node (2) at (36bp,7bp) [draw, draw = none] {$2$};
  \node (5) at (66bp,107bp) [draw, draw = none] {$5$};
  \node (4) at (66bp,57bp) [draw, draw = none] {$4$};
  \node (7) at (6bp,107bp) [draw, draw = none] {$7$};
  \node (6) at (36bp,107bp) [draw, draw = none] {};
  \node (9) at (21bp,157bp) [draw, draw = none] {$9$};
  \node (8) at (96bp,107bp) [draw, draw = none] {$8$};
  \draw [black,-] (4) ..controls (66bp,31.179bp) and (66bp,20.462bp)  .. (1);
  \draw [black,-] (9) ..controls (28.603bp,131.66bp) and (31.916bp,120.61bp)  .. (6);
  \draw [black,-] (5) ..controls (66bp,81.179bp) and (66bp,70.462bp)  .. (4);
  \draw [black,-] (3) ..controls (50.771bp,32.382bp) and (57.698bp,20.837bp)  .. (1);
  \draw [black,-] (7) ..controls (20.771bp,82.382bp) and (27.698bp,70.837bp)  .. (3);
  \draw [black,-] (3) ..controls (36bp,31.179bp) and (36bp,20.462bp)  .. (2);
  \draw [black,-] (9) ..controls (13.397bp,131.66bp) and (10.084bp,120.61bp)  .. (7);
  \draw [black,-] (8) ..controls (81.229bp,82.382bp) and (74.302bp,70.837bp)  .. (4);
  \draw [black,-] (6) ..controls (36bp,81.179bp) and (36bp,70.462bp)  .. (3);
\end{tikzpicture}%
\end{center}
\\ \hline
\textbf{Step 2 (continued):} The minimal element that covered $6$ was $9$, so replace $6$ with $9$. &
\textbf{Step 3:} Since $9$ was a local maximum, replace $9$ with $10$. & 
\textbf{Step 4:} Decrease all labels by $1$.  The resulting linear extension is $\partial \pi$.
\\ \hline
\begin{center}
\begin{tikzpicture}[>=latex,line join=bevel,]
\node (1) at (66bp,7bp) [draw, draw = none] {$3$};
  \node (3) at (36bp,57bp) [draw, draw = none] {$6$};
  \node (2) at (36bp,7bp) [draw, draw = none] {$2$};
  \node (5) at (66bp,107bp) [draw, draw = none] {$5$};
  \node (4) at (66bp,57bp) [draw, draw = none] {$4$};
  \node (7) at (6bp,107bp) [draw, draw = none] {$7$};
  \node (6) at (36bp,107bp) [draw, draw = none] {$9$};
  \node (9) at (21bp,157bp) [draw, draw = none] {};
  \node (8) at (96bp,107bp) [draw, draw = none] {$8$};
  \draw [black,-] (4) ..controls (66bp,31.179bp) and (66bp,20.462bp)  .. (1);
  \draw [black,-] (9) ..controls (28.603bp,131.66bp) and (31.916bp,120.61bp)  .. (6);
  \draw [black,-] (5) ..controls (66bp,81.179bp) and (66bp,70.462bp)  .. (4);
  \draw [black,-] (3) ..controls (50.771bp,32.382bp) and (57.698bp,20.837bp)  .. (1);
  \draw [black,-] (7) ..controls (20.771bp,82.382bp) and (27.698bp,70.837bp)  .. (3);
  \draw [black,-] (3) ..controls (36bp,31.179bp) and (36bp,20.462bp)  .. (2);
  \draw [black,-] (9) ..controls (13.397bp,131.66bp) and (10.084bp,120.61bp)  .. (7);
  \draw [black,-] (8) ..controls (81.229bp,82.382bp) and (74.302bp,70.837bp)  .. (4);
  \draw [black,-] (6) ..controls (36bp,81.179bp) and (36bp,70.462bp)  .. (3);
\end{tikzpicture}%
\end{center}
&
\begin{center}
\begin{tikzpicture}[>=latex,line join=bevel,]
\node (1) at (66bp,7bp) [draw, draw = none] {$3$};
  \node (3) at (36bp,57bp) [draw, draw = none] {$6$};
  \node (2) at (36bp,7bp) [draw, draw = none] {$2$};
  \node (5) at (66bp,107bp) [draw, draw = none] {$5$};
  \node (4) at (66bp,57bp) [draw, draw = none] {$4$};
  \node (7) at (6bp,107bp) [draw, draw = none] {$7$};
  \node (6) at (36bp,107bp) [draw, draw = none] {$9$};
  \node (9) at (21bp,157bp) [draw, draw = none] {$10$};
  \node (8) at (96bp,107bp) [draw, draw = none] {$8$};
  \draw [black,-] (4) ..controls (66bp,31.179bp) and (66bp,20.462bp)  .. (1);
  \draw [black,-] (9) ..controls (28.603bp,131.66bp) and (31.916bp,120.61bp)  .. (6);
  \draw [black,-] (5) ..controls (66bp,81.179bp) and (66bp,70.462bp)  .. (4);
  \draw [black,-] (3) ..controls (50.771bp,32.382bp) and (57.698bp,20.837bp)  .. (1);
  \draw [black,-] (7) ..controls (20.771bp,82.382bp) and (27.698bp,70.837bp)  .. (3);
  \draw [black,-] (3) ..controls (36bp,31.179bp) and (36bp,20.462bp)  .. (2);
  \draw [black,-] (9) ..controls (13.397bp,131.66bp) and (10.084bp,120.61bp)  .. (7);
  \draw [black,-] (8) ..controls (81.229bp,82.382bp) and (74.302bp,70.837bp)  .. (4);
  \draw [black,-] (6) ..controls (36bp,81.179bp) and (36bp,70.462bp)  .. (3);
\end{tikzpicture}%
\end{center}
&
\begin{center}
\begin{tikzpicture}[>=latex,line join=bevel,]
\node (1) at (66bp,7bp) [draw, draw = none] {$2$};
  \node (3) at (36bp,57bp) [draw, draw = none] {$5$};
  \node (2) at (36bp,7bp) [draw, draw = none] {$1$};
  \node (5) at (66bp,107bp) [draw, draw = none] {$4$};
  \node (4) at (66bp,57bp) [draw, draw = none] {$3$};
  \node (7) at (6bp,107bp) [draw, draw = none] {$6$};
  \node (6) at (36bp,107bp) [draw, draw = none] {$8$};
  \node (9) at (21bp,157bp) [draw, draw = none] {$9$};
  \node (8) at (96bp,107bp) [draw, draw = none] {$7$};
  \draw [black,-] (4) ..controls (66bp,31.179bp) and (66bp,20.462bp)  .. (1);
  \draw [black,-] (9) ..controls (28.603bp,131.66bp) and (31.916bp,120.61bp)  .. (6);
  \draw [black,-] (5) ..controls (66bp,81.179bp) and (66bp,70.462bp)  .. (4);
  \draw [black,-] (3) ..controls (50.771bp,32.382bp) and (57.698bp,20.837bp)  .. (1);
  \draw [black,-] (7) ..controls (20.771bp,82.382bp) and (27.698bp,70.837bp)  .. (3);
  \draw [black,-] (3) ..controls (36bp,31.179bp) and (36bp,20.462bp)  .. (2);
  \draw [black,-] (9) ..controls (13.397bp,131.66bp) and (10.084bp,120.61bp)  .. (7);
  \draw [black,-] (8) ..controls (81.229bp,82.382bp) and (74.302bp,70.837bp)  .. (4);
  \draw [black,-] (6) ..controls (36bp,81.179bp) and (36bp,70.462bp)  .. (3);
\end{tikzpicture}
\end{center}

\\ \hline
\end{tabular}
\caption{Constructing $\pi\partial$ from $\pi$.}
\label{figure:promotion-example2}
\end{figure}
\end{center}
\end{eg}

We now generalize this to {\bf extended promotion}, whose seed is any of the numbers $1,2,\ldots,n$. The algorithm 
is similar to the original one, and we describe it for seed $j$. Start with the subposet $P_{j}$ and perform 
steps 1--3 in a completely analogous fashion. Now decrease all the labels strictly larger than $j$ by 1 in $P$ 
(not only $P_{j}$). Clearly this gives a new linear extension, which we denote $\pi \partial_{j}$.
Note that $\partial_n$ is always the identity.

The extended promotion operator can be expressed in terms of more elementary operators $\tau_i$ ($1\le i<n$) as shown 
in~\cite{haiman.1992, malvenuto_reutenauer.1994,stanprom} and has explicitly been used to count linear 
extensions in~\cite{edelman_hibi_stanley.1989}. 
Let $\pi=\pi_1 \cdots \pi_n \in \L(P)$ be a linear extension of a finite poset $P$ in one-line notation. Then
\begin{equation} \label{deftau}
	\pi \tau_i = \begin{cases}
	\pi_1 \cdots \pi_{i-1} \pi_{i+1} \pi_i \cdots \pi_n & \text{if $\pi_i$ and $\pi_{i+1}$ are not}\\
	& \text{comparable in $P$,}\\
	\pi_1 \cdots \pi_n & \text{otherwise.} \end{cases}
\end{equation}
Alternatively, $\tau_i$ acts non-trivially on a linear extension if interchanging entries $\pi_i$ and $\pi_{i+1}$ yields another 
linear extension. Then as an operator on $\L(P)$,
\begin{equation}
\label{equation.promotion tau}
 	\partial_j = \tau_j \tau_{j+1} \cdots \tau_{n-1}.
\end{equation}

\subsection{Properties of $\tau_i$ and extended promotion}
The operators $\tau_i$ are involutions ($\tau_i^2 = 1$) and partially commute ($\tau_i \tau_j = \tau_j \tau_i$ when $|i-j|>1$). 
Unlike the generators for the symmetric group, the $\tau_i$ do not always satisfy the braid relation 
$\tau_i \tau_{i+1} \tau_i = \tau_{i+1} \tau_i \tau_{i+1}$. They do, however, satisfy $(\tau_i \tau_{i+1})^6 = 1$~\cite{stanprom}.

\begin{prop}\label{tau.braid.relations}
Let $P$ be a poset on $[n]$.  The braid relations 
\[
	\pi\tau_j\tau_{j+1}\tau_j = \pi\tau_{j+1}\tau_j\tau_{j+1}
\]
hold for all $1 \leq j < n-1$ and all $\pi \in \L(P)$ if and only if $P$ is a union of disjoint chains.
\end{prop}

The proof is an easy case-by-case check. Since we do not use this result, we omit the proof.

It will also be useful to express the operators $\tau_i$ in terms of the generalized promotion operator.

\begin{lem}
\label{lemma.tau_partial}
For all $1 \leq j \leq n-1$, each operator $\tau_j$ can be expressed as a product of promotion operators.
\end{lem}

\begin{proof}
We prove the claim by induction on $j$, starting with the case that $j=n-1$ and decreasing until we reach the case that $j=1$.  When $j=n-1$, the claim is obvious since $\tau_{n-1} = \partial_{n-1}$.  For $j < n-1$, we observe that 
\begin{eqnarray*}
\tau_j &=& \tau_{j}\tau_{j+1}\cdots \tau_{n-1} \tau_{n-1}\cdots \tau_{j+2}\tau_{j+1} \\
&=& \partial_j \tau_{n-1} \cdots \tau_{j+2} \tau_{j+1}.
\end{eqnarray*}
By our inductive hypothesis, each of $\tau_{j+1},\ldots,\tau_{n-1}$ can be expressed as a product of promotion operators, and hence so too can $\tau_j$.
\end{proof}

\section{Various Markov chains} 
\label{section.markov chains}

We now consider various discrete-time Markov chains related to the extended promotion operator. 
For completeness, we briefly review the part of the theory relevant to us.

Fix a finite poset $P$ of size $n$. The operators $\{\tau_i \mid 1\le
i<n\}$ (resp. $\{\partial_i \mid 1\le i\le n\}$), define a directed
graph on the set of linear extensions $\L(P)$. The vertices of the
graph are the elements in $\L(P)$ and there is an edge from $\pi$ to
$\pi'$ if $\pi' = \pi \tau_i$ (resp. $\pi' = \pi\partial_i$).  We can
now consider random walks on this graph with probabilities given
formally by $x_{1},\dots,x_n$ which sum to 1. In each case we give two
ways to assign the edge weights, see Sections~\ref{subsection.tau
  uniform}--\ref{subsection.promotion}.  An edge with weight $x_{i}$
is traversed with that rate.  A priori, the $x_{i}$'s must be positive
real numbers for this to make sense according to the standard
techniques of Markov chains.  However, the ideas work in much greater
generality and one can think of this as an ``analytic continuation.''

A discrete-time Markov chain is defined by the {\bf transition matrix}
$M$, whose entries are indexed by elements of the state space. In our
case, they are labeled by elements of $\L(P)$. We take the convention
that the $(\pi',\pi)$ entry gives the probability of going from $\pi
\to \pi'$. The special case of the diagonal entry at $(\pi,\pi)$ gives
the probability of a loop at the $\pi$.  This ensures that column sums
of $M$ are one and consequently, one is an eigenvalue with row (left-)
eigenvector being the all-ones vector.  A Markov chain is said to be
{\bf irreducible} if the associated digraph is strongly connected. In
addition, it is said to be {\bf aperiodic} if the greatest common
divisor of the lengths of all possible loops from any state to itself
is one.  For irreducible aperiodic chains, the Perron-Frobenius
theorem guarantees that there is a unique {\bf stationary
  distribution}. This is given by the entries of the column (right-)
eigenvector of $M$ with eigenvalue 1. Equivalently, the stationary distribution
$w(\pi)$ is the solution of the {\bf master equation}, given by
\be \label{master.equation} 
\sum_{\pi' \in \L(P)} M_{\pi,\pi'} \;w(\pi') 
= \sum_{\pi' \in \L(P)} M_{\pi',\pi} \; w(\pi).  
\ee 
Edges which are loops contribute to both sides equally and thus cancel
out.  For more on the theory of finite state Markov chains, see
\cite{levin_peres_wilmer.2009}.

We set up a running example that will be used for each
case. Appendix~\ref{section.appendix} shows how to define and work
with this poset in Sage.

\begin{eg} \label{example.running example}
Define $P$ by its covering relations $\{ (1,3), (1,4), (2,3) \}$, so that its Hasse diagram is as shown below:
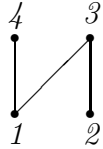
\setlength{\unitlength}{1mm}
\begin{center}
\begin{picture}(20, 20)
\put(10,4){\circle*{1}}
\put(20,4){\circle*{1}}
\put(9,0){1}
\put(19,0){2}
\put(10,14){\circle*{1}}
\put(20,14){\circle*{1}}
\put(9,16){4}
\put(19,16){3}
\put(10,4){\line(0,1){10}}
\put(20,4){\line(0,1){10}}
\put(10,4){\line(1,1){10}}
\end{picture}
\end{center}

Then the elements of
$
\L(P) = \{ 1234, 1243, 1423, 2134, 2143 \}
$
are represented by the following diagrams respectively:
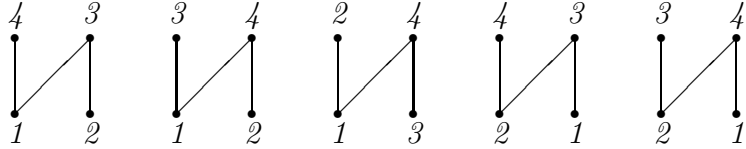
\begin{center}
\begin{picture}(20, 20)
\put(10,4){\circle*{1}}
\put(20,4){\circle*{1}}
\put(9,0){1}
\put(19,0){2}
\put(10,14){\circle*{1}}
\put(20,14){\circle*{1}}
\put(9,16){4}
\put(19,16){3}
\put(10,4){\line(0,1){10}}
\put(20,4){\line(0,1){10}}
\put(10,4){\line(1,1){10}}
\end{picture}
\begin{picture}(20, 20)
\put(10,4){\circle*{1}}
\put(20,4){\circle*{1}}
\put(9,0){1}
\put(19,0){2}
\put(10,14){\circle*{1}}
\put(20,14){\circle*{1}}
\put(9,16){3}
\put(19,16){4}
\put(10,4){\line(0,1){10}}
\put(20,4){\line(0,1){10}}
\put(10,4){\line(1,1){10}}
\end{picture}
\begin{picture}(20, 20)
\put(10,4){\circle*{1}}
\put(20,4){\circle*{1}}
\put(9,0){1}
\put(19,0){3}
\put(10,14){\circle*{1}}
\put(20,14){\circle*{1}}
\put(9,16){2}
\put(19,16){4}
\put(10,4){\line(0,1){10}}
\put(20,4){\line(0,1){10}}
\put(10,4){\line(1,1){10}}
\end{picture}
\begin{picture}(20, 20)
\put(10,4){\circle*{1}}
\put(20,4){\circle*{1}}
\put(9,0){2}
\put(19,0){1}
\put(10,14){\circle*{1}}
\put(20,14){\circle*{1}}
\put(9,16){4}
\put(19,16){3}
\put(10,4){\line(0,1){10}}
\put(20,4){\line(0,1){10}}
\put(10,4){\line(1,1){10}}
\end{picture}
\begin{picture}(20, 20)
\put(10,4){\circle*{1}}
\put(20,4){\circle*{1}}
\put(9,0){2}
\put(19,0){1}
\put(10,14){\circle*{1}}
\put(20,14){\circle*{1}}
\put(9,16){3}
\put(19,16){4}
\put(10,4){\line(0,1){10}}
\put(20,4){\line(0,1){10}}
\put(10,4){\line(1,1){10}}
\end{picture}
\end{center}
\end{eg}

\subsection{Uniform transposition graph}
\label{subsection.tau uniform}

The vertices of the {\bf uniform transposition graph} are the elements in $\L(P)$ and there is an 
edge between $\pi$ and $\pi'$ if and only if $\pi' = \pi \tau_j$ for some $j \in [n]$, where we define
$\tau_n$ to be the identity map. This edge is assigned
the symbolic weight $x_{j}$. The name ``uniform" is motivated by the fact that the stationary distribution of
this Markov chain turns out to be uniform. Note that this chain is more general than the chains considered
in~\cite{karzanov_khachiyan.1991} in that we assign arbitrary weights $x_j$ on the edges.

\begin{eg} \label{example.tau uniform}
Consider the poset and linear extensions of Example~\ref{example.running example}.
The uniform transposition graph is illustrated in Figure~\ref{figure.tau uniform}. 
\begin{figure}[h]
\begin{center}
\begin{tikzpicture} [>=triangle 45]
\draw (0,4) node {\verb!1234!};
\draw (0,0) node {\verb!1243!};
\draw (6,0) node {\verb!2134!};
\draw (6,4) node {\verb!2143!};
\draw (3,2) node {\verb!1423!};
\draw [<->] (6,.3) -- (6,3.7);
\draw (6.3,2) node {$x_3$};
\draw [<->] (0,.3) -- (0,3.7);
\draw (-.3,2) node {$x_3$};
\draw [<->] (.3,.3) -- (2.5,1.75);
\draw (1.6,.8) node {$x_2$};
\draw [<->] (.5,0) to [out = 0, in = -90] (5.7,3.7);
\draw (3,.1) node {$x_1$};
\draw [<->] (.5,4) to [out = 0, in = 90] (5.7,.3);
\draw (3,3.9) node {$x_1$};
\end{tikzpicture}
\caption{Uniform transposition graph for Example~\ref{example.running
    example}.  Every vertex has four outgoing edges labeled $x_1$ to $x_4$
  and self-loops are not drawn.
\label{figure.tau uniform}}
\end{center}
\end{figure}
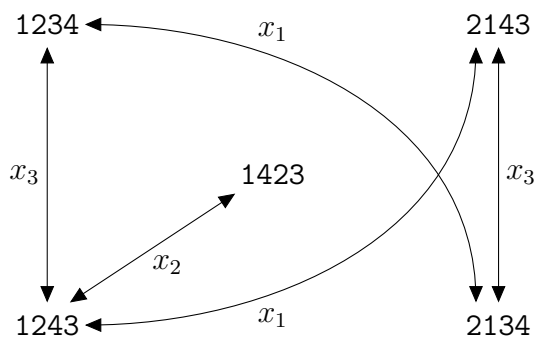
The transition matrix,  with the lexicographically ordered basis, is given by
\[
\begin{pmatrix}
x_{2}+x_{4} & x_{3}  & 0 & x_{1} & 0 \\
x_{3} & x_{4} & x_{2} & 0 & x_{1} \\
0 & x_{2} & x_{1}+x_{3}+x_{4} & 0 & 0 \\
x_{1} & 0 & 0 & x_{2}+x_{4} & x_{3} \\
0 & x_{1} & 0 & x_{3} & x_{2}+x_{4}
\end{pmatrix}.
\]
Note that the weight $x_4$ only appears on the diagonal since $\tau_4$
acts as the identity for $n=4$.  By construction, the column sums of
the transition matrix are one. Note that in this example the row sums
are also one (since the matrix is symmetric), which means that the
stationary state of this Markov chain is uniform. We will prove this
in general in Theorem~\ref{theorem.uniform transposition}.
\end{eg}

\subsection{Transposition graph}
\label{subsection.tau}

The {\bf transposition graph} is defined in the same way as the uniform transposition graph, except that 
the edges are given the symbolic weight $x_{\pi_j}$ whenever $\tau_{j}$ takes $\pi \to \pi'$.

\begin{eg} \label{example.tau}
The transposition graph for the poset in Example~\ref{example.running example} is
illustrated in Figure~\ref{figure.tau}. 
\begin{figure}[h]
\begin{center}
\begin{tikzpicture} [>=triangle 45]
\draw (0,4) node {\verb!1234!};
\draw (0,0) node {\verb!1243!};
\draw (6,0) node {\verb!2134!};
\draw (6,4) node {\verb!2143!};
\draw (3,2) node {\verb!1423!};
\draw [->] (6.3,.3) to [out=70, in = -70] (6.3,3.7);
\draw (7,2) node {$x_3$};
\draw [->] (6,3.7) -- (6,.3);
\draw (6.3,2) node {$x_4$};
\draw [->] (-.3,.3) to [out=110, in = -110] (-.3,3.7);
\draw (-1,2) node {$x_4$};
\draw [->] (0,3.7) -- (0,.3);
\draw (-.3,2) node {$x_3$};
\draw [->] (.5,.2) to [out = 30, in = 240] (2.75, 1.75);
\draw (1.5,1) node {$x_2$};
\draw [->] (2.5,2) to [out = 210, in = 60] (.3,.3);
\draw (1.5,1.7) node {$x_4$};
\draw [->] (.5,-.2) to [out = 0, in = -90] (5.8,3.7);
\draw (3,3.95) node {$x_1$};
\draw [->] (5.6,3.7) to [out = -90, in = 0] (.5,0);
\draw (3,3.35) node {$x_2$};
\draw [->] (.5,4.2) to [out = 0, in = 90] (5.8,.3);
\draw (3, .65) node {$x_2$};
\draw [->] (5.6,.3) to [out = 90, in = 0] (.5,4);
\draw (3,0) node {$x_1$};
\end{tikzpicture}
\caption{Transposition graph for Example~\ref{example.running
    example}. Every vertex has four outgoing edges labeled $x_1$ to
  $x_4$ and self-loops are not drawn.
\label{figure.tau}}
\end{center}
\end{figure}
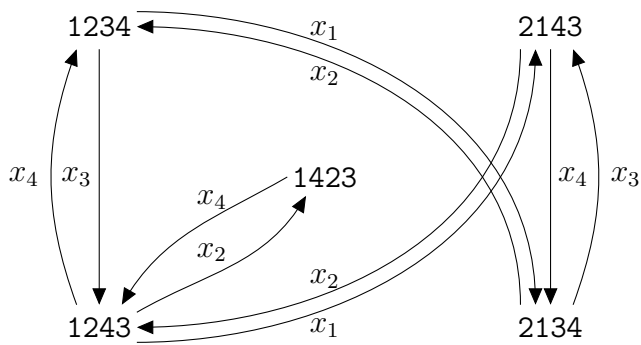
The transition matrix is given by
\be
\begin{pmatrix}
x_{2}+x_{4} & x_{4}  & 0 & x_{2} & 0 \\
x_{3} & x_{3} & x_{4} & 0 & x_{2} \\
0 & x_{2} & x_{1}+x_{2}+x_{3} & 0 & 0 \\
x_{1} & 0 & 0 & x_{1}+x_{4} & x_{4} \\
0 & x_{1} & 0 & x_{3} & x_{1}+x_{3}
\end{pmatrix}.
\ee
Again, by definition the column sums are one, but the row sums are not one in this example.
In fact, the stationary distribution (column vector with eigenvalue 1) is given by the eigenvector
\be
\begin{pmatrix}
1, &
\ds {\frac {x_{{3}}}{x_{{4}}}}, &
\ds {\frac {x_{{2}}x_{{3}}}{{x_{{4}}}^{2}}}, &
\ds {\frac {x_{{1}}}{x_{{2}}}}, &
\ds {\frac {x_{{1}}x_{{3}}}{x_{{2}}x_{{4}}}}
\end{pmatrix}^{T} \; .
\ee
We give a closed form expression for the weights of the stationary distribution in the general case
in Theorem~\ref{theorem.transposition}.
\end{eg}

\subsection{Uniform promotion graph}
\label{subsection.promotion uniform}

The vertices of the {\bf uniform promotion graph} are labeled by elements of $\L(P)$ and there is an 
edge between $\pi$ and $\pi'$ if and only if $\pi' = \pi \partial_{j}$ for some $j \in [n]$. In this case, the edge 
is given the symbolic weight $x_{j}$.

\begin{eg}
The uniform promotion graph for the poset in Example~\ref{example.running example} 
is illustrated in Figure~\ref{figure.uniform promotion}.
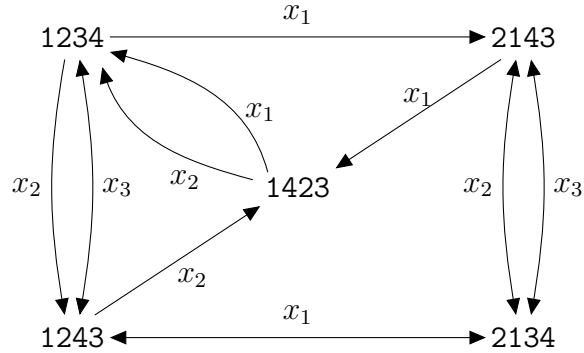
\begin{figure}[h]
\begin{center}
\begin{tikzpicture} [>=triangle 45]
\draw (0,4) node {\verb!1234!};
\draw (0,0) node {\verb!1243!};
\draw (6,0) node {\verb!2134!};
\draw (6,4) node {\verb!2143!};
\draw (3,2) node {\verb!1423!};
\draw [<->] (.1,.3) to[out=80, in=-80] (.1,3.7);
\draw (.6,2) node {$x_3$};
\draw [->] (.5,4) -- (5.5,4);
\draw (3,4.3) node {$x_1$};
\draw [->] (-.1,3.7) to[out=260, in=100] (-.1,.3);
\draw (-.6,2) node {$x_2$};
\draw [<->] (.5,0) -- (5.5,0);
\draw (3,.3) node {$x_1$};
\draw [->] (.3,.3) -- (2.5,1.75);
\draw (1.6,.8) node {$x_2$};
\draw [<->] (6.1,.3) to [out=80, in = -80] (6.1,3.7);
\draw (6.6,2) node {$x_3$};
\draw [<->] (5.9,.3) to [out = 100, in = -100] (5.9,3.7);
\draw (5.4,2) node {$x_2$};
\draw [->] (5.7,3.7) -- (3.5,2.25);
\draw (4.6,3.2) node {$x_1$};
\draw [->] (2.6,2.2) to[out=105, in = -15] (.5,3.8);
\draw (2.5,3) node {$x_1$};
\draw [->] (2.4,2.1) to[out=165, in = -75] (.4,3.6);
\draw (1.5,2.1) node {$x_2$};
\end{tikzpicture}
\end{center}
\caption{Uniform promotion graph for Example~\ref{example.running
    example}. Every vertex has four outgoing edges labeled $x_1$ to
  $x_4$ and self-loops are not drawn.
\label{figure.uniform promotion}}
\end{figure}
The transition matrix, with the lexicographically ordered basis, is given by
\[
\begin{pmatrix}
x_{4} & x_{3}  & x_{1}+x_{2} & 0 & 0 \\
x_{2}+x_{3} & x_{4} & 0 & x_{1} & 0 \\
0 & x_{2} & x_{3}+x_{4} & 0 & x_{1} \\
0 & x_{1} & 0 & x_{4} & x_{2}+x_{3} \\
x_{1} & 0 & 0 & x_{2} + x_{3} & x_{4}
\end{pmatrix}\;.
\]
Note that as in Example~\ref{example.tau uniform} the row sums are one
although the matrix is not symmetric, so that the stationary state of
this Markov chain is uniform. We prove this for general finite posets
in Theorem~\ref{theorem.uniform promotion}.

As in the uniform transposition graph, $x_{4}$ occurs only on the
diagonal in the above transition matrix.  This is because the action
of $\partial_{4}$ (or in general $\partial_n$) maps every linear
extension to itself resulting in a loop.
\end{eg}

\subsection{Promotion graph}
\label{subsection.promotion}

The {\bf promotion graph} is defined in the same fashion as the uniform promotion graph with the exception that
the edge between $\pi$ and $\pi'$ when $\pi' = \pi \partial_{j}$ is given the weight $x_{\pi_j}$. 

\begin{eg} \label{example.promotion}
The promotion graph for the poset of Example~\ref{example.running example}
is illustrated in Figure~\ref{figure.promotion}. Although it might appear that there are many more
edges here than in Figure~\ref{figure.uniform promotion}, this is not the case.
\begin{figure}[h]
\begin{center}
\begin{tikzpicture} [>=triangle 45]
\draw (0,4) node {\verb!1234!};
\draw (0,0) node {\verb!1243!};
\draw (6,0) node {\verb!2134!};
\draw (6,4) node {\verb!2143!};
\draw (3,2) node {\verb!1423!};
\draw [->] (.1,.3) to[out=80, in=-80] (.1,3.7);
\draw (.6,2) node {$x_4$};
\draw [->] (-.1,3.7) to[out=260, in=100] (-.1,.3);
\draw (-.6,2) node {$x_2$};
\draw [->] (-.3,3.7) to[out=225, in = 135] (-.3,.3);
\draw (-1.3,2) node {$x_3$};
\draw [->] (.5,4) -- (5.5,4);
\draw (3,4.3) node {$x_1$};
\draw [->] (.5,.1) to [out = 15, in = 165] (5.5,.1);
\draw (3,.7) node {$x_1$};
\draw [->] (5.5,-.1) to [out = -165, in = -15] (.5,-.1);
\draw (3,-.7) node {$x_2$};
\draw [->] (.3,.3) -- (2.5,1.75);
\draw (1.6,.8) node {$x_2$};
\draw [->] (6.1,.3) to [out=80, in = -80] (6.1,3.7);
\draw (6.6,2) node {$x_3$};
\draw [->] (6.3,3.7) to [out = -45, in = 45] (6.3,.3);
\draw (7.3,2) node {$x_4$};
\draw [<->] (5.9,.3) to [out = 100, in = -100] (5.9,3.7);
\draw (5.4,2) node {$x_1$};
\draw [->] (5.7,3.7) -- (3.5,2.25);
\draw (4.6,3.2) node {$x_2$};
\draw [->] (2.6,2.2) to[out=105, in = -15] (.5,3.8);
\draw (2.5,3) node {$x_1$};
\draw [->] (2.4,2.1) to[out=165, in = -75] (.4,3.6);
\draw (1.5,2.1) node {$x_4$};
\end{tikzpicture}
\caption{Promotion graph for Example~\ref{example.running
    example}. Every vertex has four outgoing edges labeled $x_1$ to
  $x_4$ and self-loops are not drawn.
\label{figure.promotion}}
\end{center}
\end{figure}
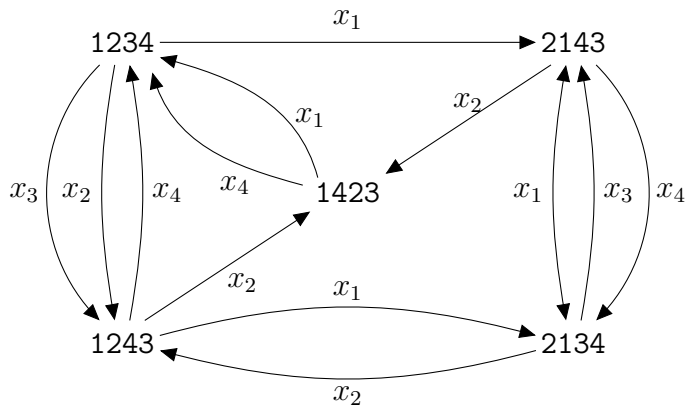
The transition matrix this time is given by
\[
\begin{pmatrix}
x_{4} & x_{4}  & x_{1}+x_{4} & 0 & 0 \\
x_{2}+x_{3} & x_{3} & 0 & x_{2} & 0 \\
0 & x_{2} & x_{2}+x_{3} & 0 & x_{2} \\
0 & x_{1} & 0 & x_{4} & x_{1}+x_{4} \\
x_{1} & 0 & 0 & x_{1} + x_{3} & x_{3}
\end{pmatrix}\;.
\]
Notice that row sums are no longer one. The stationary distribution,
as a vector written in row notation is
\[
\begin{pmatrix}
1, &
\ds \frac{x_{1}+x_{2}+x_{3}}{x_{1}+x_{2}+x_{4}}, &
\ds\frac{(x_{1}+x_{2})(x_{1}+x_{2}+x_{3})}
{(x_{1}+x_{2})(x_{1}+x_{2}+x_{4})}, &
\ds \frac{x_{1}}{x_{2}}, &
\ds \frac{x_{1}(x_{1}+x_{2}+x_{3})}
{x_{2}(x_{1}+x_{2}+x_{4})} 
\end{pmatrix}^{T}\; .
\]
Again, we will give a general such result in Theorem~\ref{theorem.promotion}.
\end{eg}

In Appendix~\ref{section.appendix}, implementations of these Markov
chains in {\tt Sage} and {\tt Maple} are discussed.

\section{Properties of the various Markov chains} \label{section.properties}

In Section~\ref{subsection.irreducible} we prove that the Markov chains defined in
Section~\ref{section.markov chains} are all irreducible.
This is used in Section~\ref{subsection.stationary} to conclude that their
stationary state is unique and either uniform or given by an explicit
product formula in their weights.

Throughout this section we fix a poset $P$ of size $n$
and let $\L:=\L(P)$ be the set of its linear extensions.

\subsection{Irreducibility}
\label{subsection.irreducible}

We now show that the four graphs of Section~\ref{section.markov chains} are all {\bf strongly connected}.

\begin{prop}
\label{proposition.strongly_connected}
Consider the digraph $G$ whose vertices are labeled by elements of $\L$ and whose edges are given as follows: 
for $\pi, \pi' \in \L$, there is an edge between $\pi$ and $\pi'$ in $G$ if and only if 
$\pi' = \pi \partial_{j}$ (resp. $\pi' = \pi \tau_j$) for some $j \in [n]$ (resp. $j\in [n-1]$). Then $G$ is strongly connected.
\end{prop}

\begin{proof}
We begin by showing the statement for the generalized promotion operators $\partial_j$.
From an easy generalization of \cite{stanprom}, we see that extended promotion, given by 
$\partial_{j}$, is a bijection for any $j$. Therefore, every element of $\L$ has exactly one such edge pointing in and 
one such edge pointing out. Moreover, $\partial_j$ has finite order, so that $\pi \partial_{j}^{k} = \pi$ for some $k$. 
In other words, the action of $\partial_{j}$ splits $\L$ into disjoint cycles. In particular, $\pi \partial_{n} = \pi$ 
for all $\pi$ so that it decomposes $\L$ into cycles of size 1.

It suffices to show that there is a directed path from any $\pi$ to 
the identity $e$.  We prove this by induction on $n$. The case of the poset with a single element is vacuous.
Suppose the statement is true for every poset of size $n-1$. We have two cases. First, suppose $\pi^{-1}_{1}=1$. In this case 
$\partial_{2},\dots,\partial_{n}$ act on $\L$ in exactly the same way as $\partial_{1},\dots,\partial_{n-1}$ on $\L'$, the set 
of linear extensions of $P'$, the poset obtained from $P$ by removing 1. Then the directed path exists by the induction 
assumption.

Instead suppose $\pi^{-1}_{1}=j$ and $\pi^{-1}_{k}=1$, for $j,k>1$. 
In other words, the label $j$ is at position 1 and label $1$ is at position $k$ of $P$.
Since $j$ is at the position of a minimal element in $P$, it does not belong to the upper set of 1 (that is $j \not \succeq 1$ 
in the relabeled poset). Thus, the only effect on $j$ of applying 
$\partial_{1}$ is to reduce it by 1, i.e., if $\pi'=\pi \partial_{1}$, then 
$\pi'^{-1}_{1}=j-1$. Continuing this way, we can get to the previous case by the action of $\partial_{1}^{j-1}$ on $\pi$.

The statement for the $\tau_j$ now follows from Lemma~\ref{lemma.tau_partial}.
\end{proof}

\begin{cor}
Assuming that the edge weights are strictly positive, all Markov
chains of Section~\ref{section.markov chains} are irreducible and
their stationary distribution is unique.
\end{cor}

\begin{proof}
Since the underlying graph of all four Markov chains of
Section~\ref{section.markov chains} is strongly connected, they are
irreducible. The existence of a single loop at any vertex of the graph
guarantees aperiodicity. The uniqueness of the stationary distribution
then follows by standard theory of Markov
chains~\cite[Chapter 1]{levin_peres_wilmer.2009}.
\end{proof}

\subsection{Stationary states}
\label{subsection.stationary}

In this section we prove properties of the stationary state of the various discrete-time Markov chains defined in 
Section~\ref{section.markov chains}, assuming that all $x_i$'s are strictly positive.

\begin{thm} \label{theorem.uniform promotion}
The discrete-time Markov chain according to the uniform promotion graph
has the uniform stationary distribution, that is, each linear extension is equally likely to occur.
\end{thm}

\begin{proof}
Stanley showed~\cite{stanprom} that the promotion operator has finite order, that is $\partial^k=\mathrm{id}$
for some $k$. The same arguments go through for the extended promotion operators $\partial_j$. Therefore
at each vertex $\pi\in \L(P)$, there is an incoming and outgoing edge corresponding to $\partial_j$ for each $j\in [n]$.  For the 
uniform promotion graph, an edge for $\partial_j$ is assigned weight $x_j$, and hence the row sum of the 
transition matrix is one, which proves the result. Equivalently, the all ones vector is the required eigenvector.
\end{proof}

\begin{thm} \label{theorem.uniform transposition}
The discrete-time Markov chain according to the uniform transposition graph
has the uniform stationary distribution.
\end{thm}

\begin{proof}
Since each $\tau_{j}$ is an involution, every incoming edge with weight $x_{j}$ has an outgoing edge with the 
same weight. Another way of saying the same thing is that the transition matrix is symmetric. By definition, the 
transition matrix is constructed so that column sums are one. Therefore, row sums are also one. 
\end{proof}

We now turn to the promotion and transposition graphs of Section~\ref{section.markov chains}.
In this case we find nice product formulas for the stationary weights.

\begin{thm} \label{theorem.promotion}
The stationary state weight $w(\pi)$ of the linear extension $\pi\in \L(P)$ for the discrete-time Markov chain for the 
promotion graph is given by
\be \label{formulaII}
w(\pi) = \prod_{i=1}^{n} \frac{x_{1}+ \cdots + x_{i}}
{x_{\pi_1}+ \cdots + x_{\pi_i}}\;,
\ee
assuming $w(e)=1$.
\end{thm}

\begin{rem} \label{remark.constant}
The entries of $w$ do not, in general, sum to one. Therefore this
is not a true probability distribution, but this is easily remedied by a multiplicative constant $Z_{P}$ depending 
only on the poset $P$.
\end{rem}

\begin{proof}[Proof of Theorem~\ref{theorem.promotion}]
We prove the theorem by induction on $n$. The case $n=1$ is trivial. 
By Remark~\ref{remark.constant}, it suffices to prove the result
for any normalization of $w(\pi)$. 
For our purposes it is most convenient to use the normalization
\be \label{equation.w}
w(\pi) = \prod_{i=1}^{n} \frac{1}{x_{\pi_1}+ \cdots + x_{\pi_i}}.
\ee
To prove~\eqref{equation.w}, we need to show that it satisfies the master 
equation \eqref{master.equation}, rewritten as
\be \label{equation.master}
	w(\pi) \left(\sum_{i=1}^n x_{\pi_i} \right) = 
\sum_{\substack{j=1 \\ \pi'=\pi \tau_{n-1} \cdots \tau_j}}^n x_{\pi'_j} w(\pi').
\ee
The left-hand side is the contribution of the outgoing edges, whereas
the right-hand side gives the weights of the incoming edges of vertex $\pi$.

Singling out the term $j=n$ and setting $\tilde{\pi} := \pi \tau_{n-1}$, the 
right-hand side of~\eqref{equation.master} becomes
\be
x_{\pi_n} w(\pi) + 
\sum_{\substack{j=1 \\\pi' = \tilde{\pi} \tau_{n-2} \cdots \tau_j}}^{n-1} x_{\pi'_j} w(\pi').
\ee
Now, notice that the $n$-th entry of $\pi'$ in one-line notation in every term of the sum is
$\tilde\pi_n$ which is either $\pi_n$ or $\pi_{n-1}$. Let $\tilde\sigma$ be 
considered as a permutation of size $n-1$ given by $(\tilde\pi_1, \dots, 
\tilde\pi_{n-1})$. Then using the formula for $w$ in~\eqref{equation.w} to separate out
the last term in the product, we obtain
\be
\sum_{\substack{j=1 \\\pi' = \tilde{\pi} \tau_{n-2} \cdots \tau_j}}^{n-1} x_{\pi'_j} w(\pi')
= \frac 1{x_{\pi_1}+ \cdots + x_{\pi_n}}
\sum_{\substack{j=1 \\\sigma' = \tilde{\sigma} \tau_{n-2} \cdots \tau_j}}^{n-1} x_{\sigma'_j} 
w(\sigma')
\ee
The induction assumption now applies to the sum on the right hand side and 
hence~\eqref{equation.master} yields
\begin{align*}
&x_{\pi_n} w(\pi) + 
\sum_{\substack{j=1 \\\pi' = \tilde{\pi} \tau_{n-2} \cdots \tau_j}}^{n-1} x_{\pi'_j} w(\pi')\\
= &x_{\pi_n} w(\pi) + \frac 1{x_{\pi_1}+ \cdots + x_{\pi_n}} 
w(\tilde{\sigma}) (x_{\tilde{\pi}_1} + \cdots + x_{\tilde{\pi}_{n-1}}), \\
= &x_{\pi_n} w(\pi) + w(\tilde{\pi}) (x_{\tilde{\pi}_1} + \cdots + x_{\tilde{\pi}_{n-1}}).
\end{align*}

We now distinguish two cases: either $\tau_{n-1}$ acts trivially on $\pi$ or not. In the first case, set $\tilde{\pi} = \pi$ and 
we immediately obtain the left-hand side of~\eqref{equation.master}. In the second case, observe that $w(\pi)$ as 
in~\eqref{equation.w} satisfies the following recursion if $\tau_j$ acts non-trivially
\[
		w(\pi \tau_j) = \frac{x_{\pi_1} + \cdots + x_{\pi_j}}{x_{\pi_1} + \cdots + x_{\pi_{j-1}} + x_{\pi_{j+1}}} w(\pi).
\]
Using this for $j=n-1$ and $x_{\tilde{\pi}_1} + \cdots + x_{\tilde{\pi}_{n-1}} = x_{\pi_1} + \cdots + x_{\pi_{n-2}} + x_{\pi_n}$ 
yields the left-hand side of~\eqref{equation.master}.
\end{proof}

When $P$ is the $n$-antichain, then $\L = S_{n}$. In this case, the
probability distribution of Theorem~\ref{theorem.promotion} has been
studied in a completely different context by Hendricks
\cite{hendricks1,hendricks2} and is known in the literature as the
\textbf{Tsetlin library}~\cite{tsetlin}, which we now
describe. Suppose that a library consists of $n$ books
$b_{1},\dots,b_{n}$ on a single shelf. Assume that only one book is
picked at a time and is returned before the next book is picked
up. The book $b_{i}$ is picked with probability $x_{i}$ and placed at
the end of the shelf.

We now explain why promotion on the $n$-antichain is the Tsetlin
library. A given ordering of the books can be identified with a permutation
$\pi$. The action of $\partial_k$ on $\pi$ gives $\pi \tau_k \cdots
\tau_{n-1}$ by \eqref{equation.promotion tau}, where now all the
$\tau_i$'s satisfy the braid relation since none of the $\pi_j$'s are
comparable. Thus the $k$-th element in $\pi$ is moved all the way to
the end. The probability with which this happens is $x_{\pi_k}$, which
makes this process identical to the action of the Tsetlin library.

The stationary distribution of the Tsetlin library
is a special case of Theorem~\ref{theorem.promotion}. In this case, $Z_{P}$ of Remark~\ref{remark.constant} also 
has a nice product formula, leading to the probability distribution,
\be 
w(\pi) = \prod_{i=1}^{n} \frac{x_{\pi_{i}}}{x_{\pi_1}+ \cdots + x_{\pi_i}}.
\ee
Letac~\cite{letac} considered generalizations of the Tsetlin library to rooted trees (meaning that each element in $P$
besides the root has precisely one successor). Our results hold for any finite poset $P$.

\begin{thm} \label{theorem.transposition}
The stationary state weight $w(\pi)$ of the linear extension $\pi \in \L(P)$ of the transposition graph 
is given by
\be \label{equation.w tau}
w(\pi) = \prod_{i=1}^{n} x_{\pi_{i}}^{i-\pi_{i}} \;,
\ee
assuming $w(e)=1$.
\end{thm}

\begin{proof}
To prove the above result, we need to show that it satisfies the
master equation \eqref{master.equation}, rewritten as
\be \label{equation.master.transp} w(\pi) \Bigl(\sum_{i=1}^n x_{\pi_i}
\Bigr) = \sum_{j=1}^n x_{\pi^{(j)}_j} w(\pi^{(j)}), \ee where
$\pi^{(j)}=\pi \tau_j$. Let us compare $\pi^{(j)}$ and $\pi$. By
definition, they differ at the positions $j$ and $j+1$ at most. Either
$\pi^{(j)}=\pi$, or $\pi^{(j)}_{j} = \pi_{j+1}$ and $\pi^{(j)}_{j+1} =
\pi_{j}$.  In the former case, we get a contribution to the right hand
side of~\eqref{equation.master.transp} of $x_{\pi_{j}} w(\pi)$,
whereas in the latter, $x_{\pi_{j+1}} w(\pi^{(j)})$.  But note that in
the latter case by~\eqref{equation.w tau}
$$ \frac{w(\pi^{(j)})}{w(\pi)} = \frac{x_{\pi_{j+1}}^{j-\pi_{j+1}} x_{\pi_{j}}^{j+1-\pi_{j}} }
{x_{\pi_{j}}^{j-\pi_{j}} x_{\pi_{j+1}}^{j+1-\pi_{j+1}} } = \frac{x_{\pi_{j}}} {x_{\pi_{j+1}}},
$$
and the contribution is again $x_{\pi_{j}} w(\pi)$. Thus the $j$-th term on the right matches 
that on the left, and this completes the proof.
\end{proof}

\section{Partition functions and eigenvalues for rooted
forests} \label{section.chains}

For a certain class of posets, we are able to give an explicit
formula for the probability distribution for the promotion graph.
Note that this involves computing the partition function $Z_P$ (see
Remark~\ref{remark.constant}). We can also specify all eigenvalues and
their multiplicities of the transition matrix explicitly.

\subsection{Main results} \label{subsection.results}
Before we can state the main theorems of this section, we need to make a couple of definitions.
A \textbf{rooted tree} is a connected poset, where each node has at most one successor. 
Note that a rooted tree has a unique largest element.
A \textbf{rooted forest} is a union of rooted trees.
A \textbf{lower set} (resp. \textbf{upper set})
$S$ in a poset is a subset of the nodes such that if $x\in S$ and $y\preceq x$ (resp. $y\succeq x$), then also $y\in S$. 
We first give the formula for the partition function.

\begin{thm} \label{theorem.partition.function}
Let $P$ be a rooted forest of size $n$ 
and let $x_{\preceq i} = \sum_{j \preceq i} x_j$. 
The partition function for the promotion graph is
given by
\be \label{equation.Zp}
Z_P = \prod_{i=1}^n \frac{x_{\preceq i}}{x_1 + \cdots + x_i}.
\ee
\end{thm}

\begin{proof}
We need to show that $w'(\pi) := Z_P \; w(\pi)$ with $w(\pi)$ given by~\eqref{formulaII} satisfies
$$\sum_{\pi \in \L(P)} w'(\pi) = 1.$$
We shall do so by induction on $n$. Assume that the formula is true for 
all rooted forests of size $n-1$. The main idea is that the last entry of $\pi$ in one-line notation has to be a maximal element 
of one of the trees in the poset. Let $P = T_1 \cup T_2 \cup \cdots \cup T_k$, where each $T_i$ is a tree. 
Moreover, let $\hat T_i$ denote the maximal element of $T_i$. Then
\begin{equation*}
\sum_{\pi \in \L(P)} w'(\pi) = \sum_{i=1}^k 
\sum_{\sigma \in \L(P \setminus \{\hat T_i\})} w'( \sigma \hat T_i )\;.
\end{equation*}
Using~\eqref{formulaII} and \eqref{equation.Zp}
\begin{equation*}
w'( \sigma  \hat T_i ) = w'(\sigma) 
\frac{x_{\preceq \hat T_i}}{x_1+\cdots+x_n},
\end{equation*}
which leads to
\begin{equation*}
\sum_{\pi \in \L(P)} w'(\pi) = \sum_{i=1}^k 
\frac{x_{\preceq \hat T_i}}{x_1+\cdots+x_n}
\sum_{\sigma \in \L(P \setminus \{\hat T_i\})} w'( \sigma ).
\end{equation*}
By the induction assumption, the rightmost sum is 1, and since each $x_j$ occurs in one and only one numerator of the sums over $i$,
an easy simplification leads to the desired result,
\end{proof}

Let $L$ be a finite poset with smallest element $\hat{0}$ and largest element $\hat{1}$.
Following~\cite[Appendix C]{brown.2000}, one may associate to each element $x\in L$ a \textbf{derangement number}
$d_x$ defined as
\begin{equation}
\label{equation.derangements}
	d_x = \sum_{y\succeq x} \mu(x,y) f([y,\hat{1}])\;,
\end{equation}
where $\mu(x,y)$ is the M\"obius function for the interval 
$[x,y] := \{z \in L \mid x\preceq z \preceq y\}$~\cite[Section 3.7]{stanenum}
and $f([y,\hat{1}])$ is the number of maximal chains in the interval $[y,\hat{1}]$.

A permutation is a \textbf{derangement} if it does not have any fixed
points.  A linear extension $\pi$ is called a \textbf{poset
  derangement} if it is a derangement when considered as a
permutation.  Let $\mathfrak d_P$ be the number of poset derangements
of the poset $P$. 

A \textbf{lattice} $L$ is a poset in which any two elements have a unique supremum (also called join) and 
a unique infimum (also called meet). For $x,y\in L$ the join is denoted by $x \vee y$, whereas the meet is
$x\wedge y$. For an \textbf{upper semi-lattice} we only require the existence of a unique supremum of
any two elements.

\begin{thm} \label{theorem.eigenvalues}
Let $P$ be a rooted forest of size $n$ and $M$ the transition matrix of
the promotion graph of Section~\ref{subsection.promotion}. Then
\begin{equation*}
	\det(M-\lambda \mathbbm{1}) = \prod_{\substack{ S \subseteq
            [n]\\ \text{$S$ upper set in $P$}}} (\lambda - x_S)^{d_S},
\end{equation*}
where $x_S = \sum_{i\in S} x_i$ and $d_S$ is the derangement number in the lattice $L$ (by inclusion) of upper
sets in $P$. In other words, for each subset $S\subseteq [n]$, which is an upper set in $P$, there is an eigenvalue
$x_S$ with multiplicity $d_S$.
\end{thm}

The proof of Theorem~\ref{theorem.eigenvalues} will be given in Section~\ref{section.R trivial}.
As we will see in Lemma~\ref{lemma.action of del}, the action of the operators in the promotion
graph of Section~\ref{subsection.promotion} for rooted forests have a Tsetlin library type interpretation 
of moving books to the end of a stack (up to reordering).

When $P$ is a union of chains, which is a special case of rooted forests, we can express the eigenvalue
multiplicities directly in terms of the number of poset derangements.

\begin{thm} \label{theorem.poset derangements}
Let $P = [n_1] + [n_2] + \cdots + [n_k]$ be a union of chains of size $n$ whose
elements are labeled consecutively within chains. Then
\begin{equation*}
	\det(M-\lambda \mathbbm{1}) = \prod_{\substack{ S \subseteq [n]\\ \text{$S$ upper set in $P$}}}
	(\lambda - x_S)^{\mathfrak d_{P \setminus S}},
\end{equation*}
where $\mathfrak{d}_\emptyset =1$.
\end{thm}

The proof of Theorem~\ref{theorem.poset derangements} is given in Section~\ref{subsection.derangement proof}.

\begin{cor}
For $P$ a union of chains, we have the identity
\be
|\L(P)| 
= \sum_{\substack{ S \subseteq [n]\\ \text{$S$ upper set in $P$}}} d_S
= \sum_{\substack{ S \subseteq [n]\\ \text{$S$ lower set in $P$}}} \mathfrak d_S \;.
\ee
\end{cor}
 
Note that the antichain is a special case of a rooted forest and in particular a union of chains. 
In this case the Markov chain is the Tsetlin library and all subsets of $[n]$ are upper (and lower) sets. Hence 
Theorem~\ref{theorem.eigenvalues} specializes to the results of Donnelly~\cite{donnelly.1991}, Kapoor and 
Reingold~\cite{kapoor_reingold.1991}, and Phatarford~\cite{phatarfod.1991} for the Tsetlin library.

The case of unions of chains, which are consecutively labeled, can be interpreted as looking at a parabolic 
subgroup of $S_n$. If there are $k$ chains of lengths $n_i$ for $1\le i \le k$, then the parabolic subgroup is
$S_{n_1} \times \cdots \times S_{n_k}$. In the realm of the Tsetlin library, there are $n_i$ books 
of the same color. The Markov chain consists of taking a book at random and placing it at the end of the stack.

\subsection{Proof of Theorem~\ref{theorem.poset derangements}}
\label{subsection.derangement proof}

We deduce Theorem~\ref{theorem.poset derangements} from
Theorem~\ref{theorem.eigenvalues} by which the matrix $M$ has
eigenvalues indexed by upper sets $S$ with multiplicity $d_{S}$. We
need to show that $\mathfrak{d}_{P \setminus S} = d_{S}$.

Let $P$ be a union of chains and $L$ the lattice of upper sets of $P$. The M\"obius function of $P$ is the 
product of the M\"obius functions of each chain. This implies that the only upper sets of $P$ with a nonzero entry 
of the M\"obius function are the ones with unions of the top element in each chain.

Since upper sets of unions of chains are again unions of chains, it suffices to consider $d_\emptyset$ for $P$
as $d_S$ can be viewed as $d_\emptyset$ for $P \setminus S$. By~\eqref{equation.derangements} we have
\[
	d_\emptyset = \sum_S \mu(\emptyset, S) f([S,\hat{1}])\; ,
\]
where the sum is over all upper sets of $P$ containing only top elements in each chain.
Recall that $f([S,\hat{1}])$ is the number of chains from $S$ to $\hat{1}$ in $L$.
By inclusion-exclusion, the claim that $d_\emptyset = \mathfrak{d}_P$ is the number of poset derangements
of $P$, that is the number of linear extensions of $P$ without fixed points, follows from the next
lemma.

\begin{lem}\label{chains-derangement}
Let $P = [n_1] + [n_2] + \cdots + [n_k]$.  Fix $I \subseteq [k]$ and let $S \subseteq P$ be the upper set 
containing the top element of the $i$th chain of $P$ for all $i \in I$.  Then $f([S,\hat{1}])$ is equal to the 
number of linear extensions of $P$ that fix at least one element of the $i$th chain of $P$ for all $i \in I$. 
\end{lem}

\begin{proof}
Let $n = n_1 + n_2 + \cdots + n_k$ denote the number of elements in $P$.  Let $N_1 = 0$ and define 
$N_i = n_1 + \cdots + n_{i-1}$ for all $2 \leq i \leq k$.  We label the elements of $P$ consecutively so that 
$N_i + 1, N_i + 2, \ldots, N_{i+1}$ label the elements of the $i$th chain of $P$ for all $1 \leq i \leq k$. 

The linear extensions of $P$ are in bijection with words $w$ of length $n$ in the alphabet 
$\mathcal{E}:= \{e_1, e_2, \ldots, e_k\}$ with $n_i$ instances of each letter $e_i$.  Indeed, given a linear 
extension $\pi$ of $P$, we associate such a word $w$ to $\pi$ by setting $w_j = e_i$ if 
$\pi_j \in \{N_i + 1, \ldots, N_{i+1}\}$; i.e. if $j$ lies in the $i$th column of $P$ under the extension $\pi$.  
For the remainder of the proof, we will identify a linear extension $\pi$ (and properties of $\pi$) with its 
corresponding word $w$. We also view $e_i$ as standard basis vectors in $\mathbb{Z}^k$.

For any $1 \leq i \leq k$ and $1 \leq j \leq n_i$, the element $N_i+j$ is fixed by $w$ if and only if $w$ 
satisfies the following two conditions: 
\begin{itemize}
\item $w_{N_i+j} = e_i$ (i.e. $w$ sends $N_i+j$ to the $i$th column of $P$) and 
\item the restriction of $w$ to its first $N_i+j$ letters, which we denote $w|_{[1,\ldots,N_i+j]}$, contains 
exactly $j$ instances of the letter $e_i$ (i.e. $N_i+j$ is the $j$th element of the $i$th column of $P$ under the extension $w$).
\end{itemize}

Moreover, it is clear that the set of all $j\in \{1,\ldots,n_i\}$ such that $w$ fixes $N_i+j$ is an interval of the form $[a_i,b_i]$.  

With $I$ and $S$ defined as in the statement of the Lemma, let 
\begin{displaymath}
n'_i:=
\begin{cases}
n_i -1 & \text{ if } i \in I, \\
n_i & \text{ if } i \notin I.
\end{cases}
\end{displaymath}
Similarly, define $N'_1 = 0$ and $N'_i = n'_1 + \cdots + n'_{i-1}$ for $i \geq 2$.  We see that $f([S,\hat{1}])$ 
counts the number of words of length $n-|I|$ in the alphabet $\mathcal{E}$ with $n'_j$ instances of each letter $e_j$.  
This is because $S$ corresponds to the element $\delta_I$ defined by
\begin{displaymath}
\delta_I(i) = 
\begin{cases}
1 & \text{ if } i \in I, \\
0 & \text{ if } i \notin I,
\end{cases}
\end{displaymath}
of $L$.  The maximal chains in $L$ from $\delta_I$ to $(n_1,n_2,\ldots,n_k)$ are lattice paths in $\mathbb{Z}^k$ 
with steps in the directions of the standard basis vectors $e_1, e_2, \ldots, e_k$.  

Having established this notation, we are ready to prove the main statement of the Lemma.  Let $\mathcal{W}$ 
denote the collection of all words in the alphabet $\mathcal{E}$ of length $n$ with $n_j$ instances of each 
letter $e_j$ that fix an element of the $i$th chain of $P$ for all $i \in I$.  Let $\mathcal{W}'$ denote the collection 
of all words of length $n-|I|$ in the alphabet $\mathcal{E}$ with $n'_j$ instances of each letter $e_j$.  

We define a bijection $\varphi: \mathcal{W} \rightarrow \mathcal{W}'$ as follows.  For each $i \in I$, suppose 
$w \in \mathcal{W}$ fixes the elements $N_i + a_i, \ldots, N_i + b_i$ from the $i$th chain of $P$.  We define 
$\varphi(w)$ to be the word obtained from $w$ by removing the letter $e_i$ in position $w_{N_i + b_i}$ for 
each $i \in I$.  Clearly $\varphi(w)$ has length  $n-|I|$ and $n'_j$ instances of each letter $e_j$.  

Conversely, given $w' \in \mathcal{W}'$, let $J_i$ be the set of indices $N'_i+j$ with $0 \leq j \leq n'_i$  such 
that $w'|_{[1,\ldots,N'_i+j]}$ contains exactly $j$ instances of the letter $e_i$.  Here we allow  $j=0$ since it 
is possible that there are no instances of the letter $e_i$ among the first $N'_i$ letters  of $w'$.  Again, it is 
clear that each $J_i$ is an interval of the form $[N'_i + c_i, \ldots, N'_i+d_i]$ and  $w'_{N_i+j} = e_i$ for all 
$j \in [c_i+1,\ldots,d_i]$, though it is possible that $w'_{N'_i+c_i} \neq e_i$.   Thus we define $\varphi^{-1}(w')$ 
to be the word obtained from $w'$ by inserting the letter $e_i$  after $w'_{N'_i+d_i}$ for all $i \in I$.  
\end{proof}

We illustrate the proof of Lemma \ref{chains-derangement} in the following example.

\begin{eg}
Let $P = [3]+[4]+[2]+[5]$, $I = \{2,4\}$, and consider the linear extension 
\[
	\pi:=1\; 10\; 4\; 8\; \mathbf{5}\; \mathbf{6}\; 2\; 3\; 11\; 9\; 7\; \mathbf{12}\; \mathbf{13}\; \mathbf{14},
\]
which corresponds to the word 
$$w = e_1 e_4 e_2 | e_3 \mathbf{e_2} \mathbf{e_2} e_1 | e_1 e_4 | e_3 e_2 \mathbf{e_4} \mathbf{e_4} \mathbf{e_4}.$$  
Here we have divided the word according to the chains of $P$.  The fixed points of $\pi$ in the second and fourth chains 
of $P$ are shown in bold, along with their corresponding entries of the word $w$.  In this case 
$\varphi(w) = e_1 e_4 e_2 e_3 e_2 e_1 e_1 e_4 e_3 e_2 e_4 e_4$. 

Conversely, consider $$w' = e_2 e_1 e_4 | e_3 e_3 e_1 | e_2 e_1 | e_2 e_4 e_4 e_4 \in \mathcal{W}'.$$  Again, we 
have partitioned $w'$ into blocks of size $n'_i$ for each $i = 1, \ldots, 4.$  In this case, $J_2 = \{4\}$ and $J_4 = \{10,11,12\}$, 
so $\varphi^{-1}(w')$ is the following word, with the inserted letters shown in bold: 
$$\varphi^{-1}(w') = e_1 e_1 e_4 | e_3 \mathbf{e_2} e_1 e_3 | e_2 e_1 | e_2 e_4 e_4 e_4 \mathbf{e_4}.$$
\end{eg}

\begin{rem}
The initial labeling of $P$ in the proof of Lemma \ref{chains-derangement} is essential to the proof.  For example, let 
$P$ be the poset $[2] + [2]$ with two chains, each of length two. Labeling the elements of $P$ so that $1<2$ and $3<4$
admits two derangements: $3142$ and $3412$.  On the other hand, labeling the elements of $P$ so that $1<4$ and $2<3$ 
only admits one derangement: $2143$.  In either case, the eigenvalue $0$ of $M$ has 
multiplicity  $2$.
\end{rem}

\section{$\R$-trivial monoids}
\label{section.R trivial}

In this section we provide the proof of Theorem~\ref{theorem.eigenvalues}. We first note that in the case of rooted forests
the monoid generated by the relabeled promotion operators of the promotion graph is $\R$-trivial
(see Sections~\ref{subsection.R trivial} and~\ref{section.R trivial promotion}). Then we use a
generalization of Brown's theory~\cite{brown.2000} for Markov chains associated to left regular bands 
(see also~\cite{bidigare.1997,bidigare_hanlon_rockmore.1999}) to $\R$-trivial monoids.
This is in fact a special case of Steinberg's results~\cite[Theorems 6.3 and 6.4]{steinberg.2006} for monoids in the
pseudovariety $\mathbf{DA}$ as stated in Section~\ref{section.brown}. The proof of Theorem~\ref{theorem.eigenvalues}
is given in Section~\ref{section.proof}.

\subsection{$\R$-trivial monoids}
\label{subsection.R trivial}

A finite \textbf{monoid} $\Monoid$ is a finite set with an associative multiplication and an identity element.
Green~\cite{green.1951} defined several preorders on $\Monoid$. 
In particular for $x,y\in \Monoid$ right and left order is defined as
\begin{equation}
\begin{split}
	x \le_{\R} y & \quad \text{if $y=xu$ for some $u\in \Monoid$,}\\
	x \le_{\LL} y & \quad \text{if $y=ux$ for some $u\in \Monoid$.}
\end{split}	
\end{equation}
(Note that this is in fact the opposite convention used by Green).
This ordering gives rise to equivalence classes ($\R$-classes or $\LL$-classes) 
\begin{equation*}
\begin{split}
	x \; \R \; y &\quad \text{if and only if $x\Monoid = y\Monoid$,}\\
	x \; \LL \; y &\quad \text{if and only if $\Monoid x = \Monoid y$.}
\end{split}
\end{equation*}
The monoid $\Monoid$ is said to be \textbf{$\R$-trivial} (resp. \textbf{$\LL$-trivial}) if all $\R$-classes (resp. $\LL$-classes)
have cardinality one.

\begin{rem}
A monoid $\Monoid$ is a left regular band if $x^2=x$ and $xyx=xy$ for all $x,y\in \Monoid$.
It is not hard to check (see also~\cite[Example 2.4]{berg_bergeron_bhargava_saliola.2011}) that
left regular bands are $\R$-trivial.
\end{rem}

Schocker~\cite{schocker.2008} introduced the notion of weakly ordered monoids which is
equivalent to the notion of $\R$-triviality~\cite[Theorem 2.18]{berg_bergeron_bhargava_saliola.2011} 
(the proof of which is based on ideas by Steinberg and Thi\'ery).

\begin{defn} \label{definition.weakly}
A finite monoid $\Monoid$ is  said to be \textbf{weakly ordered} if there is a finite upper semi-lattice $(\LM,\preceq)$
together with two maps $\supp, \des:\Monoid \to \LM$ satisfying the following axioms:
\begin{enumerate}
\item $\supp$ is a surjective monoid morphism, that is, $\supp(xy) = \supp(x) \vee \supp(y)$ for all $x,y \in \Monoid$
and $\supp(\Monoid)=\LM$.
\item If $x,y \in \Monoid$ are such that $xy\le_{\R} x$, then $\supp(y) \preceq \des(x)$.
\item If $x,y \in \Monoid$ are such that $\supp(y) \preceq \des(x)$, then $xy=x$.
\end{enumerate}
\end{defn}

\begin{thm}\cite[Theorem 2.18]{berg_bergeron_bhargava_saliola.2011}
\label{theorem.Rtrivial weakly}
Let $\Monoid$ be a finite monoid. Then $\Monoid$ is weakly ordered if and only if $\Monoid$ is $\R$-trivial.
\end{thm}

If $\Monoid$ is $\R$-trivial, then for each $x\in \Monoid$ there exists an exponent of $x$ such that
$x^\omega x = x^\omega$. In particular $x^\omega$ is idempotent, that is, $(x^\omega)^2 = x^\omega$.

Given an $\R$-trivial monoid $\Monoid$, one might be interested in finding the underlying semi-lattice $\LM$
and maps $\supp,\des$.
\begin{rem}
\label{remark.lattice}
The upper semi-lattice $\LM$ and the maps $\supp,\des$ for an $\R$-trivial monoid $\Monoid$ can be constructed as follows:
\begin{enumerate}
\item $\LM$ is the set of left ideals $\Monoid e$ generated by the idempotents $e\in \Monoid$, ordered by
reverse inclusion.
\item $\supp:\Monoid \to \LM$ is defined as $\supp(x) = \Monoid x^\omega$.
\item $\des: \Monoid \to \LM$ is defined as $\des(x)=\supp(e)$, where $e$ is some maximal element in the set
$\{y\in \Monoid \mid xy=x\}$ with respect to the preorder $\le_{\R}$.
\end{enumerate}
\end{rem}

The idea of associating a lattice (or semi-lattice) to certain monoids has been used for a long time in the 
semigroup community~\cite{clifford_preston.1961}.

\subsection{$\R$-triviality of the promotion monoid}
\label{section.R trivial promotion}

Now let $P$ be a rooted forest of size $n$ and $\prom_i$ for $1\le i\le n$ the operators on $\L(P)$
defined by the promotion graph of Section~\ref{subsection.promotion}. That is, for $\pi,\pi' \in \L(P)$,
the operator $\prom_i$ maps $\pi$ to $\pi'$ if $\pi' = \pi \partial_{\pi^{-1}_i}$.
We are interested in the monoid $\Monoid^{\prom}$ generated by $\{\prom_i \mid 1\le i \le n\}$.

\begin{lem}
\label{lemma.action of del}
Let $P$ and $\prom_i$ be as above, and $\pi \in \L(P)$. Then $\pi \prom_i$ is the linear extension in $\L(P)$
obtained from $\pi$ by moving the letter $i$ to position $n$ and reordering all letters $j \succeq i$.
\end{lem}

\begin{proof}
Suppose $\pi_i^{-1} = k$. Then the letter $i$ is in position $k$ in $\pi$. Furthermore by definition
$\pi \prom_{\pi^{-1}_i} = \pi \prom_k = \pi \tau_k \tau_{k+1} \cdots \tau_{n-1}$. Since $\pi$ is a linear
extension of $P$, all comparable letters are ordered within $\pi$. Hence $\tau_k$ either tries to switch
$i$ with a letter $j \succeq i$ or an incomparable letter $j$. In the case $j \succeq i$, $\tau_k$ acts as the
identity. In the other case $\tau_k$ switches the elements. In the first (resp. second) case we repeat the 
argument with $i$ replaced by its unique successor $j$ (resp. $i$) and $\tau_k$ replaced by $\tau_{k+1}$ etc.. 
It is not hard to see that this results in the claim of the lemma.
\end{proof}

\begin{eg}
\label{example.move letter}
Let $P$ be the union of a chain of length 3 and a chain of length 2, where the first chain is labeled by the 
elements $\{1,2,3\}$ and the second chain by $\{4,5\}$. Then $41235 \; \prom_1 = 41253$, which is obtained by moving
the letter 1 to the end of the word and then reordering the letters $\{1,2,3\}$, so that the result is again a linear extension
of $P$.

As another example, let $P$ be the rooted tree of Figure~\ref{figure.rooted tree}. Then $31245 \in \L(P)$. It is easy to check
from the definition that $31245 \; \prom_3 = 12345$. In accordance with Lemma~\ref{lemma.action of del}, we can move the letter
3 to the back to obtain $12453$. However, then the letters $3,4,5$ in $j \succeq 3$ are out of order and needs to be
reordered to obtain 12345. 
\end{eg}

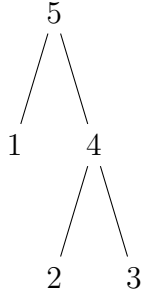
\begin{figure}
\begin{tikzpicture}[>=latex,line join=bevel,]
\node (1) at (6bp,57bp) [draw, draw = none] {$1$};
  \node (3) at (51bp,7bp) [draw, draw = none] {$3$};
  \node (2) at (21bp,7bp) [draw, draw = none] {$2$};
  \node (5) at (21bp,107bp) [draw, draw = none] {$5$};
  \node (4) at (36bp,57bp) [draw, draw = none] {$4$};
  \draw [black,-] (5) ..controls (28.603bp,81.657bp) and (31.916bp,70.612bp)  .. (4);
  \draw [black,-] (4) ..controls (28.397bp,31.657bp) and (25.084bp,20.612bp)  .. (2);
  \draw [black,-] (4) ..controls (43.603bp,31.657bp) and (46.916bp,20.612bp)  .. (3);
  \draw [black,-] (5) ..controls (13.397bp,81.657bp) and (10.084bp,70.612bp)  .. (1);
\end{tikzpicture}
\caption{Rooted tree used in Example~\ref{example.move letter}
\label{figure.rooted tree}}
\end{figure}

Let $x\in \Monoid^{\prom}$. The image of $x$ is $\im(x) = \{\pi x \mid \pi \in \L(P)\}$.
Furthermore, for each $\pi \in \im(x)$, let $\fiber(\pi,x) = \{ \pi' \in \L(P) \mid \pi = \pi' x\}$.
Let $\rfactor(x)$ be the maximal common right factor of all elements in $\im(x)$, that
is, all elements $\pi \in \im(x)$ can be written as $\pi = \pi_1 \cdots  \pi_m \rfactor(x)$ and there is no
bigger right factor for which this is true. Let us also define the set of entries in the right
factor $\Rfactor(x) = \{ i \mid i\in\rfactor(x) \}$. Note that since all elements in the image set of $x$
are linear extensions of $P$, $\Rfactor(x)$ is an upper set of $P$.

By Lemma~\ref{lemma.action of del} linear extensions in $\im(\prom_i)$ have as their last letter 
$\max_P\{j \mid j \succeq i\}$; this maximum is unique since $P$ is a rooted forest.
Hence it is clear that $\im(\prom_i x) \subseteq \im(x)$ for any $x\in \Monoid^{\prom}$ and $1\le i \le n$.
In particular, if $x\le_\LL y$, that is $y=ux$ for some $u\in \Monoid^{\prom}$, then $\im(y) \subseteq \im(x)$.
Hence $x,y$ can only be in the same $\LL$-class if $\im(x)=\im(y)$.

Fix $x \in \Monoid^{\prom}$ and let the set $I_x = \{i_1,\ldots, i_k \}$ be maximal such that 
$\prom_{i_j} x = x$ for $1\le j\le k$. The following holds.

\begin{lem} \label{claim.idempotent}
If $x$ is an idempotent, then $\Rfactor(x) = I_x$.
\end{lem}
\begin{proof}
Recall that the operators $\prom_i$ generate $\Monoid^{\prom}$. Hence we can write
$x = \prom_{\alpha_1} \cdots \prom_{\alpha_m}$ for some $\alpha_j\in [n]$. 

The condition $\prom_i x = x$ is equivalent to the condition that for every $\pi \in \im(\prom_i)$ there is a $\pi'\in \im(x)$ 
such that $\fiber(\pi, \prom_i) \subseteq \fiber(\pi',x)$ and $\pi' = \pi x$. Since $x$ is idempotent we also have $\pi' = \pi' x$.
The first condition $\fiber(\pi, \prom_i) \subseteq \fiber(\pi',x)$ makes sure that the fibers of $x$ are coarser than the fibers of $\prom_i$;
this is a necessary condition for $\prom_i x = x$  to hold (recall that we are acting on the right) since the fibers of
$\prom_i x$ are coarser than the fibers of $\prom_i$. The second condition $\pi' = \pi x$ ensures that 
$\im(\prom_i x) = \im(x)$.
Conversely, if the two conditions hold, then certainly $\prom_i x = x$.
Since $x^2=x$ is an idempotent, we hence must have $\prom_{\alpha_j} x = x$ for all $1\le j\le m$.

Now let us consider $x\prom_{\alpha_j}$. If $\alpha_j \not \in \Rfactor(x)$, then by Lemma~\ref{lemma.action of del} we have
$\Rfactor(x) \subsetneq \Rfactor(x\prom_{\alpha_j})$ and hence $|\im(x\prom_{\alpha_j})|< |\im(x)|$, which contradicts the
fact that $x^2=x$. Therefore, $\alpha_j\in \Rfactor(x)$.

Now suppose $\prom_ix=x$. Then $x=\prom_i \prom_{\alpha_1} \cdots \prom_{\alpha_m}$ and by the same arguments
as above $i\in \Rfactor(x)$. Hence $I_x \subseteq \Rfactor(x)$. 
Conversely, suppose $i\in \Rfactor(x)$. Then $x\prom_i$ has the same fibers as $x$ (but possibly a different image set since 
$\rfactor(x\prom_i) = \rfactor(x) \prom_i$ which can be different from $\rfactor(x)$). This implies $x\prom_i x =x$. 
Hence considering the expression in terms of generators $x= \prom_{\alpha_1} \cdots \prom_{\alpha_m} \prom_i  
\prom_{\alpha_1} \cdots \prom_{\alpha_m}$, the above arguments imply that $\prom_i x = x$. This shows that 
$\Rfactor(x) \subseteq I_x$ and hence $I_x = \Rfactor(x)$. This proves the claim.
\end{proof}

\begin{lem} \label{claim.nonidempotent}
$I_x$ is an upper set of $P$ for any $x\in \Monoid^{\prom}$. More precisely, $I_x = \Rfactor(e)$ for
some idempotent $e\in \Monoid^{\prom}$.
\end{lem}
\begin{proof}
For any $x\in \Monoid^{\prom}$, $\rfactor(x) \subseteq \rfactor(x^\ell)$ for any integer $\ell>0$. Also, the fibers of $x^\ell$
are coarser or equal to the fibers of $x$. Since the right factors can be of length at most $n$ (the size of $P$) and 
$\Monoid^{\prom}$ is finite, for $\ell$ sufficiently large we have $(x^\ell)^2=x^\ell$, so that
$x^\ell$ is an idempotent. Now take a maximal idempotent $e$ in the $\ge_{\R}$ preorder such that $ex=x$ 
(when $I_x=\emptyset$ we have $e=\mathbbm{1}$) which exists by the previous arguments. Then $I_e=I_x$ which by
Lemma~\ref{claim.idempotent} is also $\Rfactor(e)$. This proves the claim.
\end{proof}

Let $M$ be the transition matrix of the promotion graph of
Section~\ref{subsection.promotion}. Define $\Monoid$ to be the monoid
generated by $\{ G_i \mid 1\le i \le n\}$, where $G_i$ is the matrix
$M$ evaluated at $x_i =1$ and all other $x_j=0$.  We are now ready to
state the main result of this section.

\begin{thm} \label{theorem.r trivial}
$\Monoid$ is $\R$-trivial.
\end{thm}

\begin{rem} \label{remark.L trivial}
Considering the matrix monoid $\Monoid$ is equivalent to considering the abstract monoid $\Monoid^{\prom}$ 
generated by $\{\prom_i \mid 1\le i \le n\}$. Since the operators $\prom_i$ act on the right on linear extensions, the 
monoid $\Monoid^{\prom}$ is $\LL$-trivial instead of $\R$-trivial.
\end{rem}

\begin{eg} \label{example.r trivial monoid}
Let $P$ be the poset on three elements $\{1,2,3\}$, where $2$ covers $1$ and there are no further relations.
The linear extensions of $P$ are $\{123,132,312\}$. 
The monoid $\Monoid$ with $\R$-order, where an edge labeled $i$ means right multiplication by $G_i$,
is depicted in Figure~\ref{figure.monoid}. From the picture it is clear that the elements in the monoid are partially ordered.
This confirms Theorem~\ref{theorem.r trivial} that the monoid is $\R$-trivial.
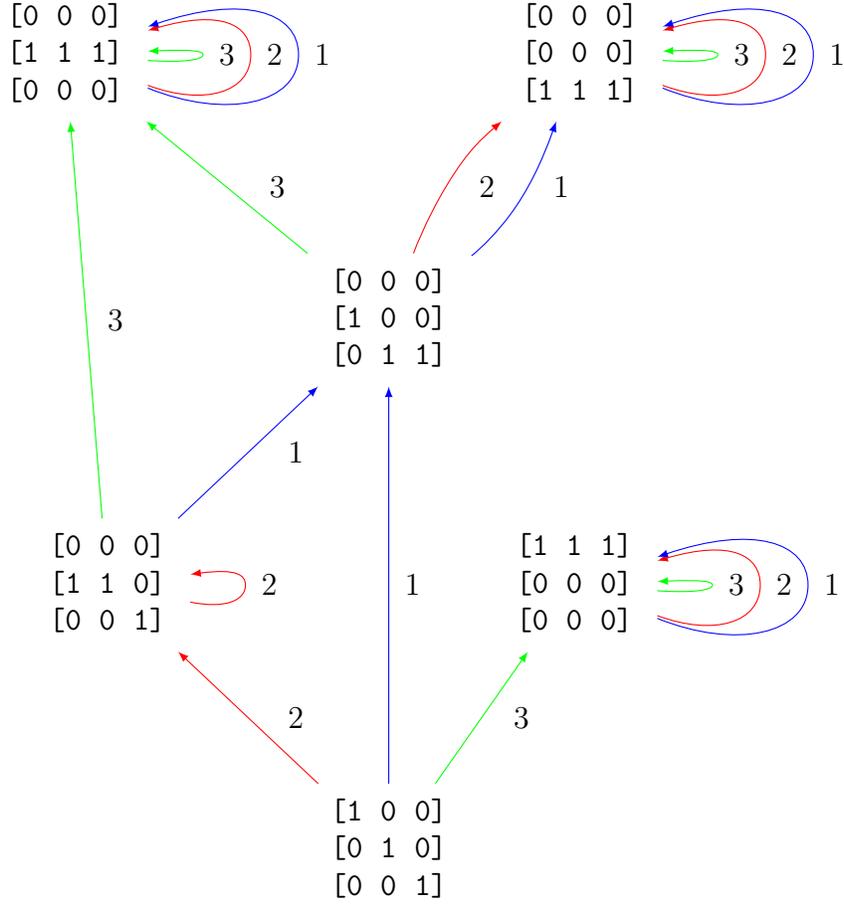
\begin{figure}
\begin{tikzpicture}[>=latex,line join=bevel]
\node (000111000) at (34bp,322bp) [draw,draw=none] {$\begin{array}{l}\verb|[0|\phantom{x}\verb|0|\phantom{x}\verb|0]|\\\verb|[1|\phantom{x}\verb|1|\phantom{x}\verb|1]|\\\verb|[0|\phantom{x}\verb|0|\phantom{x}\verb|0]|\end{array}$};
  \node (000100011) at (156bp,222bp) [draw,draw=none] {$\begin{array}{l}\verb|[0|\phantom{x}\verb|0|\phantom{x}\verb|0]|\\\verb|[1|\phantom{x}\verb|0|\phantom{x}\verb|0]|\\\verb|[0|\phantom{x}\verb|1|\phantom{x}\verb|1]|\end{array}$};
  \node (100010001) at (156bp,22bp) [draw,draw=none] {$\begin{array}{l}\verb|[1|\phantom{x}\verb|0|\phantom{x}\verb|0]|\\\verb|[0|\phantom{x}\verb|1|\phantom{x}\verb|0]|\\\verb|[0|\phantom{x}\verb|0|\phantom{x}\verb|1]|\end{array}$};
  \node (111000000) at (226bp,122bp) [draw,draw=none] {$\begin{array}{l}\verb|[1|\phantom{x}\verb|1|\phantom{x}\verb|1]|\\\verb|[0|\phantom{x}\verb|0|\phantom{x}\verb|0]|\\\verb|[0|\phantom{x}\verb|0|\phantom{x}\verb|0]|\end{array}$};
  \node (000110001) at (50bp,122bp) [draw,draw=none] {$\begin{array}{l}\verb|[0|\phantom{x}\verb|0|\phantom{x}\verb|0]|\\\verb|[1|\phantom{x}\verb|1|\phantom{x}\verb|0]|\\\verb|[0|\phantom{x}\verb|0|\phantom{x}\verb|1]|\end{array}$};
  \node (000000111) at (228bp,322bp) [draw,draw=none] {$\begin{array}{l}\verb|[0|\phantom{x}\verb|0|\phantom{x}\verb|0]|\\\verb|[0|\phantom{x}\verb|0|\phantom{x}\verb|0]|\\\verb|[1|\phantom{x}\verb|1|\phantom{x}\verb|1]|\end{array}$};
  \draw [green,<-] (111000000) ..controls (274.92bp,124.11bp) and (278bp,123.24bp)  .. (278bp,122bp) .. controls (278bp,120.01bp) and (270.13bp,118.97bp)  .. (111000000);
  \definecolor{strokecol}{rgb}{0.0,0.0,0.0};
  \pgfsetstrokecolor{strokecol}
  \draw (287bp,122bp) node {$3$};
  \draw [red,<-] (111000000) ..controls (284.71bp,139.19bp) and (296bp,133.4bp)  .. (296bp,122bp) .. controls (296bp,107.97bp) and (278.89bp,102.44bp)  .. (111000000);
  \draw (305bp,122bp) node {$2$};
  \draw [blue,<-] (111000000) ..controls (293.25bp,144.84bp) and (314bp,138.26bp)  .. (314bp,122bp) .. controls (314bp,103.42bp) and (286.9bp,97.471bp)  .. (111000000);
  \draw (323bp,122bp) node {$1$};
  \draw [red,<-] (000110001) ..controls (98.925bp,128.34bp) and (102bp,125.72bp)  .. (102bp,122bp) .. controls (102bp,116.04bp) and (94.127bp,112.92bp)  .. (000110001);
  \draw (111bp,122bp) node {$2$};
  \draw [green,<-] (000111000) ..controls (82.925bp,324.11bp) and (86bp,323.24bp)  .. (86bp,322bp) .. controls (86bp,320.01bp) and (78.127bp,318.97bp)  .. (000111000);
  \draw (95bp,322bp) node {$3$};
  \draw [red,<-] (000111000) ..controls (92.706bp,339.19bp) and (104bp,333.4bp)  .. (104bp,322bp) .. controls (104bp,307.97bp) and (86.892bp,302.44bp)  .. (000111000);
  \draw (113bp,322bp) node {$2$};
  \draw [blue,<-] (000111000) ..controls (101.25bp,344.84bp) and (122bp,338.26bp)  .. (122bp,322bp) .. controls (122bp,303.42bp) and (94.897bp,297.47bp)  .. (000111000);
  \draw (131bp,322bp) node {$1$};
  \draw [red,<-] (000000111) ..controls (190.08bp,289.62bp) and (186.84bp,285.85bp)  .. (184bp,282bp) .. controls (175.41bp,270.34bp) and (168.49bp,255.62bp)  .. (000100011);
  \draw (193bp,272bp) node {$2$};
  \draw [blue,<-] (000000111) ..controls (213.54bp,281.1bp) and (208.34bp,270.66bp)  .. (202bp,262bp) .. controls (197.25bp,255.51bp) and (191.26bp,249.35bp)  .. (000100011);
  \draw (221bp,272bp) node {$1$};
  \draw [red,<-] (000110001) ..controls (97.168bp,77.502bp) and (117.64bp,58.193bp)  .. (100010001);
  \draw (121bp,72bp) node {$2$};
  \draw [blue,<-] (000100011) ..controls (156bp,148.97bp) and (156bp,79.056bp)  .. (100010001);
  \draw (165bp,122bp) node {$1$};
  \draw [green,<-] (000111000) ..controls (39.818bp,249.27bp) and (45.43bp,179.13bp)  .. (000110001);
  \draw (53bp,222bp) node {$3$};
  \draw [green,<-] (111000000) ..controls (194.12bp,76.461bp) and (181.06bp,57.8bp)  .. (100010001);
  \draw (206bp,72bp) node {$3$};
  \draw [green,<-] (000000111) ..controls (276.92bp,324.11bp) and (280bp,323.24bp)  .. (280bp,322bp) .. controls (280bp,320.01bp) and (272.13bp,318.97bp)  .. (000000111);
  \draw (289bp,322bp) node {$3$};
  \draw [red,<-] (000000111) ..controls (286.71bp,339.19bp) and (298bp,333.4bp)  .. (298bp,322bp) .. controls (298bp,307.97bp) and (280.89bp,302.44bp)  .. (000000111);
  \draw (307bp,322bp) node {$2$};
  \draw [blue,<-] (000000111) ..controls (295.25bp,344.84bp) and (316bp,338.26bp)  .. (316bp,322bp) .. controls (316bp,303.42bp) and (288.9bp,297.47bp)  .. (000000111);
  \draw (325bp,322bp) node {$1$};
  \draw [green,<-] (000111000) ..controls (87.65bp,278.02bp) and (111.61bp,258.39bp)  .. (000100011);
  \draw (114bp,272bp) node {$3$};
  \draw [blue,<-] (000100011) ..controls (108.83bp,177.5bp) and (88.364bp,158.19bp)  .. (000110001);
  \draw (121bp,172bp) node {$1$};
\end{tikzpicture}
\caption{Monoid $\Monoid$ in right order for the poset of Example~\ref{example.r trivial monoid}
\label{figure.monoid}}
\end{figure}
\end{eg}

\begin{eg} \label{example.not r trivial monoid}
Now consider the poset $P$ on three elements $\{1,2,3\}$, where $1$ is covered by both $2$ and $3$ with no further relations.
The linear extensions of $P$ are $\{123,132\}$.
This poset is not a rooted forest. The corresponding monoid in $\R$-order is depicted in Figure~\ref{figure.monoid not r}.
The two elements
\[
	\begin{pmatrix} 0&1\\ 1&0 \end{pmatrix} \quad \text{and} \quad \begin{pmatrix} 1&0 \\ 0&1 \end{pmatrix}
\]
are in the same $\R$-class. Hence the monoid is not $\R$-trivial, which is consistent with Theorem~\ref{theorem.r trivial}.
\begin{figure}
\begin{tikzpicture}[>=latex,join=bevel]
\node (0110) at (115bp,104bp) [draw,draw=none] {$\begin{array}{l}\verb|[0|\phantom{x}\verb|1]|\\\verb|[1|\phantom{x}\verb|0]|\end{array}$};
  \node (0011) at (204bp,192bp) [draw,draw=none] {$\begin{array}{l}\verb|[0|\phantom{x}\verb|0]|\\\verb|[1|\phantom{x}\verb|1]|\end{array}$};
  \node (1001) at (114bp,16bp) [draw,draw=none] {$\begin{array}{l}\verb|[1|\phantom{x}\verb|0]|\\\verb|[0|\phantom{x}\verb|1]|\end{array}$};
  \node (1100) at (26bp,192bp) [draw,draw=none] {$\begin{array}{l}\verb|[1|\phantom{x}\verb|1]|\\\verb|[0|\phantom{x}\verb|0]|\end{array}$};
  \draw [blue,<-] (0110) ..controls (94.527bp,76.712bp) and (93.034bp,73.373bp)  .. (92bp,70bp) .. controls (89.395bp,61.501bp) and (89.482bp,58.525bp)  .. (92bp,50bp) .. controls (93.912bp,43.527bp) and (97.483bp,37.127bp)  .. (1001);
  \definecolor{strokecol}{rgb}{0.0,0.0,0.0};
  \pgfsetstrokecolor{strokecol}
  \draw (101bp,60bp) node {$1$};
  \draw [red,<-] (1100) ..controls (53.257bp,165.6bp) and (57.25bp,161.7bp)  .. (61bp,158bp) .. controls (74.031bp,145.16bp) and (88.705bp,130.47bp)  .. (0110);
  \draw (89bp,148bp) node {$2$};
  \draw [green,<-] (0011) ..controls (245.02bp,193.97bp) and (248bp,193.17bp)  .. (248bp,192bp) .. controls (248bp,190.08bp) and (239.96bp,189.17bp)  .. (0011);
  \draw (257bp,192bp) node {$3$};
  \draw [red,<-] (0011) ..controls (254.25bp,209.73bp) and (266bp,204.09bp)  .. (266bp,192bp) .. controls (266bp,177.12bp) and (248.2bp,172.01bp)  .. (0011);
  \draw (275bp,192bp) node {$2$};
  \draw [blue,<-] (0011) ..controls (262.21bp,214.9bp) and (284bp,208.83bp)  .. (284bp,192bp) .. controls (284bp,172.76bp) and (255.54bp,167.58bp)  .. (0011);
  \draw (293bp,192bp) node {$1$};
  \draw [green,<-] (1100) ..controls (67.016bp,193.97bp) and (70bp,193.17bp)  .. (70bp,192bp) .. controls (70bp,190.08bp) and (61.963bp,189.17bp)  .. (1100);
  \draw (79bp,192bp) node {$3$};
  \draw [red,<-] (1100) ..controls (76.249bp,209.73bp) and (88bp,204.09bp)  .. (88bp,192bp) .. controls (88bp,177.12bp) and (70.199bp,172.01bp)  .. (1100);
  \draw (97bp,192bp) node {$2$};
  \draw [blue,<-] (1100) ..controls (84.212bp,214.9bp) and (106bp,208.83bp)  .. (106bp,192bp) .. controls (106bp,172.76bp) and (77.542bp,167.58bp)  .. (1100);
  \draw (115bp,192bp) node {$1$};
  \draw [blue,<-] (1001) ..controls (114.46bp,56.877bp) and (114.67bp,75.378bp)  .. (0110);
  \draw (123bp,60bp) node {$1$};
  \draw [red,<-] (0011) ..controls (180.52bp,138.28bp) and (158.34bp,89.894bp)  .. (136bp,50bp) .. controls (132.56bp,43.866bp) and (128.46bp,37.336bp)  .. (1001);
  \draw (179bp,104bp) node {$2$};
  \draw [green,<-] (0011) ..controls (162.79bp,168.26bp) and (156.41bp,163.32bp)  .. (151bp,158bp) .. controls (139.59bp,146.79bp) and (129.76bp,131.4bp)  .. (0110);
  \draw (160bp,148bp) node {$3$};
  \draw [green,<-] (1100) ..controls (39.3bp,145.26bp) and (49.546bp,113.9bp)  .. (62bp,88bp) .. controls (70.658bp,69.998bp) and (74.294bp,66.186bp)  .. (86bp,50bp) .. controls (90.481bp,43.804bp) and (95.734bp,37.26bp)  .. (1001);
  \draw (71bp,104bp) node {$3$};
\end{tikzpicture}
\caption{Monoid $\Monoid$ in right order for the poset of Example~\ref{example.not r trivial monoid}
\label{figure.monoid not r}}
\end{figure}
\end{eg}

\begin{proof}[Proof of Theorem~\ref{theorem.r trivial}]
By Theorem~\ref{theorem.Rtrivial weakly} a monoid is $\R$-trivial if and only if it is weakly ordered.
We prove the theorem by explicitly constructing the semi-lattice $\LM$ and maps $\supp, \des:\Monoid^{\prom} \to \LM$ 
of Definition~\ref{definition.weakly}. In fact, since we work with $\Monoid^{\prom}$, we will establish the
left version of Definition~\ref{definition.weakly} by Remark~\ref{remark.L trivial}.

Recall that for $x \in \Monoid^{\prom}$, we defined the set $I_x = \{i_1,\ldots, i_k \}$ to be maximal such that 
$\prom_{i_j} x = x$ for $1\le j\le k$. 

Define $\des(x)=I_x$ and $\supp(x) = \des(x^\omega)$. By Lemma~\ref{claim.idempotent}, for idempotents $x$
we have $\supp(x) = \des(x) = I_x = \Rfactor(x)$. Let $\LM = \{ \Rfactor(x) \mid x\in \Monoid^{\prom}, x^2=x \}$ which 
has a natural semi-lattice structure $(\LM,\preceq)$ by inclusion of sets. The join operation is union of sets.

Certainly by Lemma~\ref{claim.idempotent} and the definition of $\LM$, the map $\supp$ is surjective. We want 
to show that in addition $\supp(xy) = \supp(x) \vee \supp(y)$, where $\vee$ is the join in $\LM$. 
Recall that $\supp(x) = \des(x^\omega) = \Rfactor(x^\omega)$.
If $x = \prom_{j_1} \cdots \prom_{j_m}$ in terms of the generators and $J_x:=\{j_1,\ldots, j_m\}$, then
by Lemma~\ref{lemma.action of del} $\Rfactor(x^\omega)$ contains the upper set of $J_x$ in $P$ plus possibly 
some more elements that are forced if the upper set of $J_x$ has only one successor in the semi-lattice of upper sets in $P$.
A similar argument holds for $y$ with $J_y$. Now again by Lemma~\ref{lemma.action of del},
$\supp(xy) = \Rfactor((xy)^{\omega})$ contains the elements in the upper set of $J_x \cup J_y$, plus possibly more
forced by the same reason as before. Hence $\supp(xy) = \supp(x) \vee \supp(y)$.
This shows that Definition~\ref{definition.weakly} (1) holds.

Suppose $x,y\in \Monoid^{\prom}$ with $yx \le_{\LL} x$. Then there exists a $z\in \Monoid^{\prom}$ such that
$zyx=x$. Hence $\supp(y) \preceq \supp(zy) \preceq I_x =\des(x)$ by Lemmas~\ref{claim.idempotent} 
and~\ref{claim.nonidempotent}.
Conversely, if $x,y\in \Monoid^{\prom}$ are such that $\supp(y) \preceq \des(x)$, then by the definition of $\des(x)$
we have $\supp(y) \preceq I_x$, which is the list of indices of the left stabilizers of $x$. By the definition of $\supp(y)$
and the proof of Lemma~\ref{claim.idempotent}, $y^\omega$ can be written as a product of $\prom_i$ with 
$i\in \supp(y)$. The same must be true for $y$. Hence $yx=x$, which shows that the left version of (2) and (3) 
of Definition~\ref{definition.weakly} hold. 

In summary, we have shown that $\Monoid^{\prom}$ is weakly ordered in $\LL$-preorder and hence $\LL$-trivial.
This implies that $\Monoid$ is $\R$-trivial.
\end{proof}

\begin{rem} \label{remark.lattice M}
In the proof of Theorem~\ref{theorem.r trivial} we explicitly constructed the semi-lattice $\LM=\{ \Rfactor(x)
\mid x\in \Monoid^{\prom}, x^2=x\}$ and the maps $\supp,\des : \Monoid^{\prom} \to \LM$ of 
Definition~\ref{definition.weakly}. Here $\des(x)= I_x$ is the set of indices $I_x=\{i_1,\ldots, i_m\}$ such that 
$\prom_{i_j} x =x$ for all $1\le j\le m$ and $\supp(x) = \des(x^\omega) = I_{x^{\omega}} = \Rfactor(x^\omega)$.
\end{rem}

\begin{eg}
\label{example.weakly ordered}
Let $P$ be the poset of Example~\ref{example.r trivial monoid}.
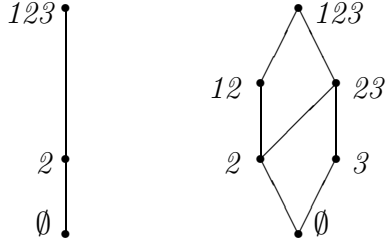
\begin{figure}
\setlength{\unitlength}{1mm}
\begin{center}
\begin{picture}(10,30)
\put(5,0){\circle*{1}}
\put(5,10){\circle*{1}}
\put(5,30){\circle*{1}}
\put(5,0){\line(0,1){30}}
\put(1,0){$\emptyset$}
\put(1,8){2}
\put(-3,28){123}
\end{picture}
\hspace{1.8cm}
\begin{picture}(10,30)
\put(5,0){\circle*{1}}
\put(0,10){\circle*{1}}
\put(10,10){\circle*{1}}
\put(0,20){\circle*{1}}
\put(10,20){\circle*{1}}
\put(5,30){\circle*{1}}
\put(5,0){\line(1,2){5}}
\put(5,0){\line(-1,2){5}}
\put(0,10){\line(0,1){10}}
\put(10,10){\line(0,1){10}}
\put(0,20){\line(1,2){5}}
\put(10,20){\line(-1,2){5}}
\put(0,10){\line(1,1){10}}
\put(7,0){$\emptyset$}
\put(-5,8){2}
\put(12,8){3}
\put(-7,18){12}
\put(12,18){23}
\put(7,28){123}
\end{picture}
\end{center}
\caption{The left graph is the lattice $\LM$ of the weakly ordered monoid for the poset in Example~\ref{example.weakly ordered}.
The right graph is the lattice $L$ of all upper sets of $P$.
\label{figure.lattice}}
\end{figure}
The monoid $\Monoid$ with $\R$-order, where an edge labeled $i$ means right multiplication by $G_i$,
is depicted in Figure~\ref{figure.monoid}. The elements $x=\mathbbm{1}, G_2,G_3,G_2G_3,G_1^2$ are
idempotent with $\supp(x)=\des(x)=\emptyset,2,123,123$, $123$, respectively. The only non-idempotent element is
$G_1$ with $\supp(G_1)=123$ and $\des(G_1)=\emptyset$.
The semi-lattice $\LM$ is the left lattice in Figure~\ref{figure.lattice}. The right graph in Figure~\ref{figure.lattice} is 
the lattice $L$ of all upper sets of $P$.
\end{eg}

\subsection{Eigenvalues and multiplicities for $\R$-trivial monoids}
\label{section.brown}

Let $\Monoid$ be a finite monoid (for example a left regular band) and $\{w_x\}_{x\in \Monoid}$ a probability
distribution on $\Monoid$ with transition matrix for the random walk given by
\begin{equation} \label{equation.transition matrix monoid}
	M(c,d) = \sum_{xc=d} w_x
\end{equation}
for $c,d \in \mathcal{C}$, where $\mathcal{C}$ is the set of maximal elements in $\Monoid$ under right
order $\ge_{\R}$. The set $\mathcal{C}$ is also called the set of \textbf{chambers}.

Recall that by Remark~\ref{remark.lattice} we can associate a semi-lattice $\LM$ and functions 
$\supp,\des:\Monoid \to \LM$ to an $\R$-trivial monoid $\Monoid$.
For $X\in \LM$,  define $c_X$ to be the number of chambers in $\Monoid_{\ge X}$, that is, the number of 
$c\in \mathcal{C}$ such that $c\ge_{\R} x$, where $x\in \Monoid$ is any fixed element with $\supp(x)=X$. 

\begin{thm} \label{theorem.main}
Let $\Monoid$ be a finite $\R$-trivial monoid with transition matrix $M$ as in~\eqref{equation.transition matrix monoid}.
Then $M$ has eigenvalues
\begin{equation}
	\lambda_X = \sum_{\substack{y \\ \supp(y) \preceq X}} w_y
\end{equation}
for each $X\in \LM$ with multiplicity $d_X$ recursively defined by
\begin{equation}
	\sum_{Y\succeq X} d_Y = c_X.
\end{equation}
Equivalently,
\begin{equation}
	d_X = \sum_{Y \succeq X} \mu(X,Y) \; c_Y,
\end{equation}
where $\mu$ is the M\"obius function on $\LM$.
\end{thm}

Brown~\cite[Theorem 4, Page 900]{brown.2000} proved Theorem~\ref{theorem.main} in the case when $\Monoid$ 
is a left regular band. Theorem~\ref{theorem.main} is a generalization to the $\R$-trivial case.
It is in fact a special case of a result of Steinberg~\cite[Theorems 6.3 and 6.4]{steinberg.2006} for monoids
in the pseudovariety $\mathbf{DA}$. This was further generalized in~\cite{steinberg.2008}.

\subsection{Proof of Theorem~\ref{theorem.eigenvalues}}
\label{section.proof}

By Theorem~\ref{theorem.r trivial} the promotion monoid $\Monoid$ is $\R$-trivial, hence
Theorem~\ref{theorem.main} applies.

Let $L$ be the lattice of upper sets of $P$ and $L^{\Monoid}$ the semi-lattice of Definition~\ref{definition.weakly}
associated to $\R$-trivial monoids that is used in Theorem~\ref{theorem.main}. Recall that for the promotion monoid
$L^{\Monoid} = \{ \Rfactor(x) \mid x\in \Monoid^{\prom}, x^2=x \}$ by Remark~\ref{remark.lattice M}.
Now pick $S\in L$ and let $r=r_1\ldots r_m$ be any linear extension of $P|_S$ (denoting $P$ restricted to $S$).
By repeated application of Lemma~\ref{lemma.action of del}, it is not hard to see that $x = \prom_{r_1}
\cdots \prom_{r_m}$ is an idempotent since $r_1 \dots r_m \subseteq \rfactor(x)$ and $x$ only acts on this
right factor and fixes it. $\rfactor(x)$ is strictly bigger than $r_1 \ldots r_m$ if some further letters beyond $r_1 \ldots r_m$
are forced in the right factors of the elements in the image set. This can only happen if there is only one successor 
$S'$ of $S$ in the lattice $L$. In this case the element in $S'\setminus S$ is forced as the letter to the left of
$r_1 \ldots r_m$ and is hence part of $\rfactor(x)$. 

Recall that $f([S,\hat{1}])$ is the number of maximal chains from $S$ to the maximal element $\hat{1}$ in $L$.
Since $L$ is the lattice of upper sets of $P$, this is precisely the number of linear extensions of $P|_{P \setminus S}$.
If $S\in L$ has only one successor $S'$, then $f([S,\hat{1}]) = f([S',\hat{1}])$. Equation~\eqref{equation.derangements}
is equivalent to
\[
	f([S,\hat{1}]) = \sum_{T \succeq S} d_T
\]
(see~\cite[Appendix C]{brown.2000} for more details). Hence $f([S,\hat{1}]) = f([S',\hat{1}])$ implies that $d_S=0$
in the case when $S$ has only one successor $S'$.

Now suppose $S\in L^{\Monoid}$ is an element of the smaller semi-lattice. Recall that $c_S$ of Theorem~\ref{theorem.main}
is the number of maximal elements in $x\in \Monoid^{\prom}$ with $x\ge_\R s$ for some $s$ with $\supp(s)=S$. 
In $\Monoid$ the maximal elements in $\R$-order (or equivalently in $\Monoid^{\prom}$ in $\LL$-order)
form the chamber $\mathcal{C}$ (resp. $\mathcal{C}^{\prom}$) and are naturally indexed by the linear extensions in $\L(P)$. 
Namely, given $\pi = \pi_1 \ldots \pi_n \in \L(P)$ the element $x = \prom_{\pi_1} \cdots \prom_{\pi_n}$ is idempotent, maximal
in $\LL$-order and has as image set $\{\pi\}$. Conversely, given a maximal element $x$ in $\LL$-order it must have
$\rfactor(x)\in \L(P)$. Given $s\in \Monoid^{\prom}$ with $\supp(s)=S$, only those maximal elements $x\in \Monoid^{\prom}$
associated to $\pi\in \im(s)$ are bigger than $s$. Hence for $S\in L^{\Monoid}$ we have $c_S = f([S,\hat{1}])$. 

The above arguments show that instead of $L^{\Monoid}$ one can also work with the lattice $L$ of upper
sets since any $S \in L$ but $S\not \in L^{\Monoid}$ comes with multiplicity $d_S=0$ and otherwise the multiplicities 
agree. 

The promotion Markov chain assigns a weight $x_i$ for a transition from $\pi$ to $\pi'$ for $\pi,\pi'\in\L(P)$ if
$\pi'=\pi\prom_i$. Recall that elements in the chamber $\mathcal{C}^{\prom}$ are naturally associated with
linear extensions. Let $x,x'\in \mathcal{C}^{\prom}$ be associated to $\pi,\pi'$, respectively. That is,
$\pi = \tau x$ and $\pi'=\tau x'$ for all $\tau\in \L(P)$. Then $x'= x \prom_i$ since $\tau (x \prom_i) 
= (\tau x) \prom_i = \pi \prom_i = \pi'$ for all $\tau \in \L(P)$. Equivalently in the monoid $\Monoid$ we would
have $X' = G_i X$ for $X,X'\in \mathcal{C}$. Hence comparing with~\eqref{equation.transition matrix monoid},
setting the probability variables to $w_{G_i}= x_i$ and $w_X=0$ for all other $X\in \Monoid$, 
Theorem~\ref{theorem.main} implies Theorem~\ref{theorem.eigenvalues}.

\begin{eg}
Figure~\ref{figure.lattice} shows the lattice $L^{\Monoid}$ on the left and the lattice $L$ of upper sets of $P$
on the right, for the monoid displayed in Figure~\ref{figure.monoid}. The elements $2,23,12$ in $L$ have only
one successor and hence do not appear in $L^{\Monoid}$.
\end{eg}

\section{Outlook} \label{section.outlook}

Two of our Markov chains, the uniform promotion graph and the uniform
transposition graph, are irreducible and have the uniform distribution
as their stationary distributions. Moreover, the former is
irreversible and has the advantage of having tunable parameters $x_1,
\dots, x_n$ whose only constraint is that they sum to 1.  Because of
the irreversibility property, it is plausible that the mixing times
for this Markov chain is smaller than the ones considered by Bubley
and Dyer~\cite{bubley.dyer.1999}. Hence the uniform promotion graph could have possible
applications for uniformly sampling linear extensions of a large
poset. This is certainly deserving of further study.  

It would also be interesting to extend the results of Brown and
Diaconis~\cite{brown_diaconis.1998} (see also~\cite{athanasiadis_diaconis.2010}) on rates of 
convergences to the Markov chains in this paper.
For the Markov chains corresponding to $\R$-trivial monoids of Section~\ref{section.chains},
one can find polynomial time exponential bounds for the rates of convergences after $\ell$ steps of the form
$c \;\ell^k \lambda^{\ell-k}$, where $c$ is the number of chambers, $\lambda = \max_i (1-x_i)$, and $k$ is 
a parameter associated to the poset. More details on rates of convergences and mixing times
can be found in~\cite{ayyer_klee_schilling.2013}.

In this paper, we have characterized posets, where the Markov chains for
the promotion graph yield certain simple formulas for their
eigenvalues and multiplicities. The eigenvalues have explicit expressions
for rooted forests and there is a concrete combinatorial interpretation for the multiplicities
as derangement numbers of permutations for unions of chains by Theorem~\ref{theorem.poset derangements}.
However, we have not covered all possible posets, whose
promotion graphs have nice properties. For example, the non-zero
eigenvalues of the transition matrix of the promotion graph of the
poset in Example~\ref{example.running example} are given by
\[
x_3+x_4, \quad x_3, \quad 0 \quad \text{and} \quad  -x_1\;,
\]
even though the corresponding monoid is not $\R$-trivial (in fact, it is not even aperiodic).
Note that the last eigenvalue is negative.
On the other hand, not all posets have this property. In particular, the
poset with covering relations $1<2,1<3$ and $1<4$ has six linear
extensions, but the characteristic polynomial of its transition matrix
does not factorize at all. It would be interesting to classify all
posets with the property that all the eigenvalues of the transition
matrices of the promotion Markov chain are linear in the probability
distribution $x_i$. In such cases, one would also like an explicit formula for
the multiplicity of these eigenvalues. 
In this paper, this was only achieved for unions of chains. Further details
are discussed in~\cite{ayyer_klee_schilling.2013}.

\appendix
\section{\texttt{Sage} and \texttt{Maple} implementations}
\label{section.appendix}

We have implemented the extended promotion and transposition operators
on linear extensions in \texttt{Maple} and also the open source
software \texttt{Sage}~\cite{sage, sage-combinat}.  The \texttt{Maple}
code is available from the homepage of one of the authors (A.A.) as
well as the preprint version on the arXiv, whereas the \texttt{Sage}
code was already integrated into \texttt{sage-5.0} (by A.S.).
Some of the figures in this paper were produced in \texttt{Sage}.

Here we illustrate how to reproduce Example~\ref{example.promotion slide} in \texttt{Sage}. 
We define the poset, view it, and create its linear extensions:
\begin{verbatim}
  sage: P = Poset(([1,2,3,4,5,6,7,8,9],
    [[1,3],[1,4],[2,3],[3,6],[3,7],[4,5],[4,8],[6,9],[7,9]]), 
    linear_extension = True)
  sage: P.show()
  sage: L = P.linear_extensions()
\end{verbatim}
Then we define the identity linear extension and compute the promotion on it:
\begin{verbatim}
  sage: pi = L([1,2,3,4,5,6,7,8,9])
  sage: pi.promotion()
  [2, 1, 4, 5, 3, 7, 8, 6, 9]
\end{verbatim}

Next we reproduce the examples of Section~\ref{section.markov chains}.
The poset and linear extensions of Example~\ref{example.running example} can be constructed
as follows:
\begin{verbatim}
  sage: P = Poset(([1,2,3,4],[[1,3],[1,4],[2,3]]))
  sage: L = P.linear_extensions()
  sage: L.list()
  [[2, 1, 3, 4], [2, 1, 4, 3], [1, 2, 3, 4], [1, 2, 4, 3], 
   [1, 4, 2, 3]]
\end{verbatim}
To compute the generalized promotion operator on this poset, using the algorithm defined in 
Section~\ref{subsection.def prom}, we first need to make sure that the poset $P$ is associated with 
the identity linear extension:
\begin{verbatim}
  sage: P = P.with_linear_extension([1,2,3,4])
\end{verbatim}
Alternatively, this is achieved via
\begin{verbatim}
  sage: P = Poset(([1,2,3,4],[[1,3],[1,4],[2,3]]), 
                    linear_extension = True)
  sage: Q = P.promotion(i=2)
  sage: Q.show()
\end{verbatim}
The various graphs of Sections~\ref{subsection.tau uniform}--\ref{subsection.promotion} can be created 
and viewed, respectively, as follows:
\begin{verbatim}
  sage: G = L.markov_chain_digraph(action='tau')
  sage: G = L.markov_chain_digraph(action='tau', 
               labeling='source')
  sage: G = L.markov_chain_digraph(action='promotion')
  sage: G = L.markov_chain_digraph(action='promotion', 
               labeling='source')
  sage: view(G)
\end{verbatim}
The transition matrices can be computed via
\begin{verbatim}
  sage: L.markov_chain_transition_matrix(action='tau')
\end{verbatim}
with again other settings for ``action" or ``labeling", depending on the desired graph.

\bibliographystyle{alpha}
\bibliography{posets}

\end{document}